\documentclass[11pt,reqno]{amsart}
\usepackage{fullpage,enumerate}
\usepackage{amsfonts}
\usepackage[all]{xy}
\usepackage{amssymb}
\usepackage{comment}
\usepackage{times}
\usepackage{graphicx}
\usepackage{amsmath,bm}
\usepackage[hidelinks]{hyperref}
\usepackage{tikz-cd}
\usepackage{tikz}
\usepackage{appendix}
\usepackage{color}\usepackage{float}
\usepackage[margin=1.3in,marginparwidth=1.1in,marginparsep=0.1in]{geometry}
\usepackage{marginnote,color}
\usepackage{import}
\usepackage[normalem]{ulem}
\usepackage{graphicx}

\usepackage{mathrsfs}
\usepackage{geometry}
\usepackage{epsfig,amscd,amsthm,mathtools,fourier}
\numberwithin{equation}{section}

\newtheorem{thm}{Theorem}[section]
\newtheorem{cor}[thm]{Corollary}
\newtheorem{lem}[thm]{Lemma}
\newtheorem{prop}[thm]{Proposition}
\newtheorem*{Thm}{Main Theorem}
\newtheorem*{ZThm}{Theorem (Zariski's Theorem)}

\theoremstyle{definition}
\newtheorem{ex}[thm]{Example}

\newtheorem{defn}[thm]{Definition}

\newcommand{\sing}{\mathrm{sing}}

\newcommand{\irr}{\mathrm{irr}}

\newcommand{\trop}{\mathrm{trop}}
\newcommand{\valence}{\mathrm{val}}

\newcommand{\Ker}{{\rm Ker}}
\newcommand{\rank}{{\rm rank}}

\newcommand{\val}{{\rm val}}

\newcommand{\Spec}{{\rm Spec}}

\newcommand{\Br}{{\rm Br}}

\newcommand{\RR}{{\mathbb R}}
\newcommand{\ZZ}{{\mathbb Z}}
\newcommand{\PP}{{\mathbb P}}
\newcommand{\NN}{{\mathbb N}}
\newcommand{\GG}{{\mathbb G}}

\newcommand{\QQ}{{\mathbb Q}}

\newcommand{\CL}{{\mathcal L}}

\newcommand{\CO}{{\mathcal O}}

\newcommand{\CM}{{\mathcal M}}
\newcommand{\CN}{{\mathcal N}}
\newcommand{\CI}{{\mathcal I}}
\newcommand{\CJ}{{\mathcal J}}

\newcommand{\CF}{{\mathcal F}}

\newcommand{\CG}{{\mathcal G}}

\newcommand{\CX}{{\mathcal X}}

\newcommand{\fm}{{\mathfrak m}}
\newcommand{\fn}{{\mathfrak n}}
\newcommand{\fG}{{\mathfrak G}}

\newcommand{\Star}{{\rm Star}}

\def\:{\colon}
\def\val{\nu}

\newcommand{\cX}{{\mathcal{X}}}

\newcommand{\pa}{\mathrm{p_a}}
\newcommand{\pg}{\mathrm{p_g}}
\newcommand{\st}{\mathrm{st}}

\newcommand{\red}{\mathrm{red}}
\newcommand{\ev}{\mathrm{ev}}
\newcommand{\ov}{\mathrm{ov}}

\newcommand{\m}{\mathfrak w}

\theoremstyle{remark}
\newtheorem{rem}[thm]{Remark}

\makeatletter
\DeclareRobustCommand{\cev}[1]{%
  {\mathpalette\do@cev{#1}}%
}
\newcommand{\do@cev}[2]{%
  \vbox{\offinterlineskip
    \sbox\z@{$\m@th#1 x$}%
    \ialign{##\cr
      \hidewidth\reflectbox{$\m@th#1\vec{}\mkern4mu$}\hidewidth\cr
      \noalign{\kern-\ht\z@}
      $\m@th#1#2$\cr
    }%
  }%
}
\makeatother

\title{On the Severi problem in arbitrary characteristic}
\author{Karl Christ, Xiang He and Ilya Tyomkin}

\thanks{IT is partially supported by the Israel Science Foundation (grant No. 821/16). KC is supported by the Israel Science Foundation (grant No. 821/16) and by the Center for Advanced Studies at BGU. XH was supported by the ERC Consolidator Grant 770922 - BirNonArchGeom.}

\address[Christ]{$~^1$Department of Mathematics\\
	Ben-Gurion University of the Negev\\P.O.Box 653 \\Be'er Sheva\\ 84105\\  Israel and $~^2$Institute of Algebraic Geometry\\
	Leibniz University Hannover\\Welfengarten 1 \\30167 Ha\-no\-ver\\  Germany }\email{kchrist@math.uni-hannover.de}

\address[He]{$~^1$Yau Mathematical Sciences Center, Ningzhai, Tsinghua University, Haidian District, Beijing, China, 100084 and $~^2$Einstein Institute of Mathematics\\
	The Hebrew University of Jerusalem\\ Giv'at Ram\\ Jerusalem\\ 91904\\ Israel}\email{xianghe@mail.tsinghua.edu.cn}

\address[Tyomkin]{Department of Mathematics\\
	Ben-Gurion University of the Negev\\P.O.Box 653 \\Be'er Sheva\\ 84105\\  Israel }\email{tyomkin@math.bgu.ac.il}

\begin{document}

\begin{abstract}
In this paper we show that Severi varieties parameterizing irreducible reduced planar curves of given degree and geometric genus are either empty or irreducible in any characteristic. Following Severi's original idea, this gives a new proof for the irreducibility of the moduli space of smooth projective curves of given genus in positive characteristic. It is the first one that does not involve a reduction to the characteristic zero case. As a further consequence, we generalize Zariski’s theorem to positive characteristic and show that a general reduced planar curve of given geometric genus is nodal.
\end{abstract}
	
\maketitle
	
\setcounter{tocdepth}{1}
\tableofcontents

\section{Introduction}\label{sec:introduction}
In the current paper, we study the geometry of Severi varieties over an algebraically closed field $K$ of arbitrary characteristic. Recall that the {\em Severi variety} $V_{g,d}$ is defined to be the locus of degree $d$ reduced plane curves of geometric genus $g$ in the complete linear system $|\CO_{\PP^2}(d)|$. We denote by $V_{g,d}^\irr$ the union of the irreducible components parameterizing irreducible curves. The main goal of this paper is to prove that the Severi varieties $V_{g,d}^\irr$ are either empty or irreducible.

These varieties were introduced by Severi in 1921, in order to prove the irreducibility of the moduli space $M_g$ of genus $g$ compact Riemann surfaces \cite{Severi}. He noticed that for $d$ large enough, there exists a natural surjective map from $V_{g,d}^\irr$ to $M_g$, since any Riemann surface of genus $g$ admits an immersion as a plane curve of degree $d$. Therefore, the irreducibility of $M_g$ follows once one proves that $V_{g,d}^\irr$ itself is irreducible.

The irreducibility (or rather connectedness) of $M_g$ was first asserted by Klein in 1882, who deduced it from the connectedness of certain spaces parameterizing branched coverings of the Riemann sphere, nowadays called {\em Hurwitz schemes}, see \cite{Klein,Hur91,EC18,Ful69}. Severi's goal was to provide an algebraic treatment of Klein's assertion. However, his proof of irreducibility of $V_{g,d}^\irr$ contained a gap, and the question, whether Severi varieties are irreducible, remained open for more than 60 years.

In 1969, Deligne and Mumford proved the irreducibility of the moduli space $\CM_g$ of smooth projective genus $g$ curves in arbitrary characteristic \cite{DM69}; see also \cite{Ful69}. The proof uses reduction to characteristic zero, and Teichm\"uller theory over $\mathbb C$. As a result, it is based on transcendental methods. Only in 1982, Fulton gave a purely algebraic proof of the irreducibility of $\CM_g$ in an appendix to the paper \cite{HM82} of Harris and Mumford, who constructed a good compactification of Hurwitz schemes. However, Fulton's proof still proceeds by a reduction to the characteristic zero case.

Although the question of irreducibility of $\CM_g$ was settled, the Severi problem remained open, and attracted quite a lot of attention. In 1982, Zariski gave a dimension-theoretic characterization of Severi varieties in characteristic zero \cite{Zar82}. Namely, he proved that if $V\subseteq |\CO_{\PP^2}(d)|$ parameterizes reduced curves of geometric genus $g$, then the dimension of $V$ is bounded by $3d+g-1$. Moreover, if equality holds, then a general curve parameterized by $V$ is necessarily nodal. 
An alternative proof of this result can be given using the techniques developed by Arbarello and Cornalba  \cite{AC81, AC81it}. Both approaches, however, fail in positive characteristic. 

Soon after, Harris settled the Severi problem in characteristic zero \cite{Har86}.
Harris' proof of irreducibility follows the general line of Severi's approach. He first proves that any component of $V_{g,d}^\irr$ contains a rational nodal curve, and then shows that there is a unique such component. The second step is relatively simple, and uses a monodromy argument. However, the first step is rather involved: it uses Zariski's theorem, a careful analysis of degenerations of curves, and a remarkable trick based on the study of the deformation spaces of tacnodes.

In positive characteristic, much less was known about the Severi problem and Zariski's theorem. In 2013, the third author proved that the bound in Zariski's theorem holds true in arbitrary characteristic. However, the assertion about the nodality of a general curve of genus $g$ was shown to be wrong, at least in the case of curves on certain weighted projective planes; see \cite{Tyo13}. In the planar case, however, it remained an intriguing open question, which we also settle in the current paper. Notice that in characteristic zero, both parts of Zariski's theorem generalize naturally to many rational surfaces, in particular to weighted projective planes, see, e.g., \cite{KS13}.

\subsection{The main result} Let $d\ge 1$ and $1-d\le g\le \binom{d-1}{2}$ be integers. We first prove that $V_{g,d}$ is locally closed (Lemma~\ref{lem:constr}), and 
equidimensional of pure dimension $3d+g-1$ (Proposition~\ref{prop:severi variety dimension}). The main result is then the following theorem, which settles the Severi problem in arbitrary characteristic.
\begin{Thm}\label{thm:irreducibility of severi variety}
Let $d\ge 1$ and $0\le g\le \binom{d-1}{2}$ be integers. Then the Severi variety $V_{g,d}^{\irr}$ is irreducible of dimension $3d+g-1$ over any algebraically closed ground field.
\end{Thm}
On the way to proving the Main Theorem, we establish Zariski's theorem in arbitrary characteristic (Corollary~\ref{cor:zariski}):

\begin{ZThm}
Let $V \subset V_{g, d}$ be an irreducible subvariety. Then,
\begin{enumerate}
\item $\dim(V)\le 3d+g-1$, and
\item if $\dim(V)=3d+g-1$, then for a general $[C]\in V$, the curve $C$ is nodal.
\end{enumerate}
\end{ZThm}

Note that unlike in the proof of Harris in characteristic zero, we do not use Zariski's theorem to establish the degeneration part of the argument, but rather obtain Zariski's theorem as a consequence of our main result about degenerations (Theorem~\ref{thm:cont_lines}).

As an immediate corollary of the main result, we obtain the first proof of the irreducibility of the moduli spaces $\CM_g$ in arbitrary characteristic that involves {\em no} reduction to characteristic zero. Indeed, since there is a dense open subset $U$ of $V_{g,d}^{\irr}$ that parameterizes nodal curves, the universal family of curves over $U$ is equinormalizable, and hence induces a map $U\to \CM_g$. However, since for any $d$ large enough, any smooth genus $g$ curve admits a birational immersion as a planar degree $d$ curve, it follows that the map $U\to \CM_g$ is dominant. Thus, the irreducibility of $\CM_g$ follows from the Main Theorem in arbitrary characteristic.

\subsection{The techniques and the idea of the proof}
Our proof of the Main Theorem also follows the general strategy of Severi's original approach, but the tools we use are completely different and combine classical algebraic geometry and tropical geometry.

As a first step, we prove that the closure $\overline V\subseteq |\CO_{\PP^2}(d)|$ of any irreducible component $V\subseteq V_{g,d}$ contains $V_{1-d,d}$ (Theorem~\ref{thm:cont_lines}). We proceed by induction, i.e., we prove that $\overline V$ contains an irreducible component of $V_{g-1,d}$, and therefore a component of $V_{1-d,d}$. But the variety $V_{1-d,d}$ is irreducible since it parameterizes unions of $d$ distinct lines, which allows us to conclude that $\overline V$ contains the whole of $V_{1-d,d}$. This is the step where the new tools from tropical geometry are involved. To do tropical geometry, we assume that the ground field is the algebraic closure of a complete discretely valued field, which is harmless because the irreducibility property is stable under field extensions, and any field is a subfield of the algebraic closure of the field of Laurent power series with coefficients in the given field. 

To prove that $\overline V$ contains an irreducible component of $V_{g-1,d}$, we consider the intersection $Z$ of $V$ with the space of curves passing through $3d+g-2$ points in general position, chosen such that their tropicalizations are {\em vertically stretched}. By a careful analysis of the tropicalization of the corresponding one-parameter family of curves of genus $g$, we show that there exists $[C]\in \overline Z$ such that $C$ is a reduced curve of geometric genus $g-1$. Since such a $C$ passes through $3d+g-2$ points in general position, it follows that the intersection $\overline V\cap V_{g-1,d}$ has the same dimension as $V_{g-1,d}$ itself, and therefore contains one of its irreducible components.


Our degeneration argument is new. It is based on the study of the tropicalizations of general one-parameter families of algebraic curves, and on the investigation of the properties of the induced map from the tropicalization of the base to the moduli space of parameterized tropical curves $\alpha\:\Lambda\to M_{g, n, \nabla}^\trop$. We prove that $\alpha$ is a piecewise integral affine map, and in good cases, it satisfies the {\em harmonicity} and {\em local combinatorial surjectivity} properties, see Definition~\ref{def:harloccombsurj} and Theorem~\ref{thm:paramtrop}. Roughly speaking, what these properties allow us to do is to describe the intersection of the image of $\alpha$ with two types of strata of $M_{g, n, \nabla}^\trop$: the {\em nice strata} -- parameterizing weightless tropical curves whose underlying graph is $3$-valent, and the {\em simple walls} -- parameterizing such curves but with a unique $4$-valent vertex. In particular, we show that for a nice stratum, the image of $\alpha$ is either disjoint from it or intersects it along a straight interval, whose boundary does not belong to the stratum. For a simple wall, we show that the image of $\alpha$ is either disjoint from it or intersects {\em all} nice strata in the star of the wall.

Next, we use these general properties to reduce the assertion of Theorem~\ref{thm:cont_lines} to a combinatorial game with floor decomposed tropical curves, the latter being a very convenient tool in tropical geometry introduced by Brugall\'e and Mikhalkin \cite{BM08}. The goal of the game is to prove that the image of $\alpha$ necessarily intersects a nice stratum parameterizing tropical curves having a contracted edge of varying length. Indeed, if this is the case, then $\Lambda$ contains a leg parameterizing tropical curves for which the length of the contracted edge is growing to infinity. Since the legs of $\Lambda$ correspond to marked points of the base, it follows that there is a closed point $[C]\in\overline Z$, such that the generalized tropicalization of the normalization of $C$ has an edge of infinite length, i.e., $C$ has geometric genus less than $g$. With a minor extra effort we show that $C$ is necessarily reduced, see Step 3 in the proof of Theorem~\ref{thm:cont_lines}.

On the tropical side of the game, we show that for a point $q\in \Lambda$ corresponding to a tropical curve passing through $3d+g-1$ vertically stretched points, the stratum $M_\Theta$ of $M_{g, n, \nabla}^\trop$ containing $\alpha(q)$ is nice, and the boundary of the interval ${\rm Im}(\alpha)\cap M_\Theta$ belongs to simple walls adjacent to $M_\Theta$. We then use the {\em floor-elevator} structure of floor decomposed curves to describe the nice strata adjacent to these simple walls, and show that by traveling along the nice strata and by crossing the corresponding simple walls, we can find a leg of $\Lambda$ such that $\alpha$ maps it to a nice stratum parameterizing tropical curves having a contracted edge of varying length, see Lemma~\ref{lem:keylemma}, and the figures within its proof.

To complete the proof of the Main Theorem we follow the ideas of \cite{Tyo07}, and consider the so called {\em decorated Severi varieties} $U_{d,\delta}$ that parameterize reduced curves with $\delta:=\binom{d-1}{2}-g$ marked nodes. We prove that these varieties are smooth, equidimensional, and that the natural map $\overline U_{d,\delta}\to |\CO_{\PP^2}(d)|$ is generically finite and has  $\overline V_{g,d}$ as its image. Finally, we show that the variety $U^\irr_{d,\delta}$ parameterizing irreducible reduced curves of degree $d$ with $\delta$ marked nodes is irreducible (Theorem~\ref{thm:irrdecsev}), which implies the Main Theorem.

The analysis of the decorated Severi varieties allows us to generalize the well-known description of the local geometry of the closure of the Severi variety $\overline V_{g,d}$ along $V_{1-d,d}$ to the case of arbitrary characteristic (Theorem~\ref{thm:branchstr}). We show that as in the case of characteristic zero, at a general $[C_0]\in V_{1-d,d}$, the variety $\overline V_{g,d}$ consists of smooth branches parameterized by all possible subsets $\mu$ of $\delta$ nodes of $C_0$. Furthermore, the intersections of branches are reduced, have expected dimension, and the branch corresponding to $\mu$ belongs to $\overline V^\irr_{g,d}$ if and only if $C_0\setminus\mu$ is connected.

\subsection{Generalizations and applications}
The ideas and the techniques introduced in this paper admit a variety of generalizations and applications, some of which we investigate in our follow-up papers. In particular, our Key Lemma (Lemma~\ref{lem:keylemma}) is a new tool providing a good control over the skeletons of curves and the way they vary as curves move in tropically general one-parameter families. Therefore we expect that further applications of our tropical machinery to the study of families of curves will be found in the near future. Let us describe two generalizations we already explored.

A natural generalization considered in \cite{CHT20b} is the case of  polarized toric surfaces corresponding to {\em $h$-transverse polygons}, e.g., Hirzebruch surfaces. These are the types of polygons for which the technique of floor-decomposed curves applies. It turns out that if a polygon $\Delta$ is $h$-transverse, then any component of the Severi variety $V^\irr_{g,\Delta}$ contains $V^\irr_{0,\Delta}$ in its closure. However, unlike in the planar case, to prove this in full generality, one has to deal with walls corresponding to elliptic components. This brings the geometry of the moduli spaces of elliptic curves with a level structure into play, and restricts the validity of the results for the general $h$-transverse polygons to characteristic zero (or large enough characteristic). In some cases, such as Hirzebruch surfaces, one can avoid elliptic walls, and prove the degeneration result in arbitrary characteristic. The degeneration result combined with the recent monodromy result of Lang \cite{lang20} gives rise to the irreducibility of the Severi varieties on a majority of such surfaces. Notice, however, that there are examples of reducible Severi varieties even on polarized toric surfaces corresponding to certain $h$-transverse polygons, see \cite{Tyo14} for a simplest example of this sort.

In \cite{CHT21b}, we make another step forward, and use a modification of the methods developed in the current paper to deduce the irreducibility of Hurwitz schemes in small positive characteristic. To the best of our knowledge, the latter has been an open question since Hurwitz schemes were introduced by Fulton in 1969 \cite{Ful69}. In particular, we develop new tropical tools that replace monodromy arguments in the proof of the irreducibility of Severi varieties.

\subsection{The structure of the paper}
The first two sections are preliminaries from algebraic and tropical geometry. In Section~\ref{sec: algebraic curves}, we define the Severi varieties and prove that they are equidimensional of expected dimension. In Section~\ref{sec: tropical curves}, we recall the notion of parameterized tropical curves, their moduli spaces, and floor decomposed curves.

Section~\ref{sec:tropicalizatoin of families of curves} is devoted to the study of the tropicalization of a general one-parameter family of smooth curves with a map to the projective plane. The results of this section are the main technical tool in the proof of the degeneration result (Theorem~\ref{thm:cont_lines}), which is proved in Section~\ref{sec:degeneration} together with the generalization of Zariski's theorem to the case of positive characteristic (Corollary~\ref{cor:zariski}).

The main result of Section~\ref{sec:local geometry of vgd} is Theorem~\ref{thm:branchstr}, which provides an explicit description of the local geometry of $\overline V_{g,d}$ along $V_{1-d,d}$. We also define and study the decorated Severi varieties in this section. Finally, in Section~\ref{sec:proofMT}, we prove the irreducibility of decorated Severi varieties, and deduce the Main Theorem.

\medskip
	
\noindent {\bf Acknowledgements.} We are grateful to Michael Temkin for helpful discussions and to an anonymous referee for many insightful comments that helped us to improve the presentation.
	
\section{Severi varieties} \label{sec: algebraic curves}
In this section, we recall the notion of Severi varieties on algebraic surfaces, and prove their basic properties. In particular, we show that in the case of toric surfaces, Severi varieties are locally closed, and their closures are either empty or equidimensional of expected dimension. Let $S$ be a projective surface over an algebraically closed ground field $K$, and $\mathcal L$ a line bundle on $S$. We write $[C] \in |\mathcal{L}|$ for the $K$-point in the complete linear system $|\mathcal{L}|$ parameterizing a curve $C$.

\begin{defn}
	The \emph{Severi varieties} $V_{g,\CL}$ and $V_{g,\CL}^{\irr}$ are defined to be the following loci in $|\mathcal L|$:
		\[
		V_{g,\CL} := \left\{[C] \in |\mathcal L|\, | \, C \text{ is reduced, } C\cap S^{\rm sing}=\emptyset,\, \pg(C) = g \right\},
		\]
where $\pg(C)$ denotes the geometric genus of $C$, and
		\[
		V_{g,\CL}^{\irr} := \left\{[C] \in V_{g,\CL}\, | \, C \text{ is irreducible}\right\}.
		\]
\end{defn}

We denote by $\overline V_{g,\CL}$ and $\overline V^\irr_{g,\CL}$ the closures in $|\mathcal L|$ of $V_{g,\CL}$ and $V^\irr_{g,\CL}$, respectively. Notice that since the locus of integral curves is open in the linear system, the variety $\overline V^\irr_{g,\CL}$ is a union of irreducible components of $\overline V_{g,\CL}$. If $(S,\CL)=\left(\PP^2,\CO_{\PP^2}(d)\right)$, we will often use the classical notation $V_{g,d}:=V_{g,\CL}$ and $V^\irr_{g,d}:=V^\irr_{g,\CL}$.

\begin{defn}
A family of curves $\cX \to B$ is called \emph{generically equinormalizable} if the normalization of the total space, $\cX^\nu \to B$, is a family of curves with smooth generic fiber.
\end{defn}
	
\begin{lem}{\cite[\S 5.10]{dJ96}} \label{lma:normalization}
Let $\cX \to B$ be a projective morphism of integral schemes with reduced curves as fibers. Then there exist maps
	\begin{equation} \label{diag:trop_comm}
		\xymatrix@=.5pc{
			\cX' \ar@{->} [rr] \ar@{->} [dd] && \cX \ar@{->} [dd]\\
			&&&&\\
			B' \ar@{->} [rr] && B\\
		}
    \end{equation}
such that $B' \to B$ is finite and surjective, $\cX'$ is an irreducible component of $(\cX \times_B B')_{\red}$ dominating $\cX$ and $\cX' \to B'$ is a generically equinormalizable family of reduced projective curves.
\end{lem}

\begin{rem} \label{rem: equinormalizable}
In characteristic zero, any family $\cX\to B$ like in Lemma~\ref{lma:normalization} is generically equinormalizable by \cite[Page 80]{Tei80} and \cite{CL06}. This fails in characteristic $p$ as the classic example below shows. Furthermore, the map $B'\to B$ of Lemma~\ref{lma:normalization} may fail to be generically \'etale, which is one of the pitfalls in a na\"ive attempt to generalize the known characteristic-zero approaches to Zariski's theorem and Harris' theorem to the case of positive characteristic.
\end{rem}
	
\begin{ex}
In characteristic $p > 2$, consider the family over $\mathbb{A}^1_K = \Spec(K[t])$, given in an affine chart $\Spec(K[x, y, t])$ by $x^p + y^2 + t = 0$. It has normal total space isomorphic to $\mathbb{A}^2_K$. However, none of its fibers is smooth. In this case, the necessary base change in Lemma~\ref{lma:normalization} is given by adjoining $\sqrt[p]{t}$. In particular, the base change is nowhere \'etale.
\end{ex}
	
\begin{lem}\label{lem:constr}
The Severi varieties $V_{g,\CL}^{\irr}$ and $V_{g,\CL}$ are locally closed subsets of $|\CL|$.
\end{lem}

\begin{proof}
Since $V_{g,\CL}^{\irr}$ is the intersection of $V_{g,\CL}$ with the open locus of irreducible curves in $|\mathcal L|$, it is sufficient to prove the assertion for $V_{g,\CL}$. Since the universal curve $\cX_{|\CL|} \to |\CL|$ is flat and proper, the set $U_1 := \left\{[C] \in |\CL| | \, C \text{ is reduced } \right\} $ is open in  $|\CL|$; see \cite[\href{https://stacks.math.columbia.edu/tag/0C0E}{Tag~0C0E}]{stacks-project}.

Let us first show that $V_{g,\CL}$ is constructible. We can inductively define a finite stratification of $U_1$ by geometric genera as follows: Consider the restriction $\cX_{U_1} \to U_1$ of the universal curve to $U_1$. Assume first, that the total space $\cX_{U_1}$ is irreducible. Let $\cX_{U_1}' \to U_1'$ be a family associated to $\cX_{U_1} \to U_1$ as in Lemma~\ref{lma:normalization}. Then there is a dense open subset $V_1'$ of $U'_1$ over which $\cX_{U_1}' \to U'_1$ has constant geometric genus. Since the map $U_1'\to U_1$ is finite, the image $Z$ of the complement of $V_1'$ is closed and nowhere dense in $U_1$. Set $V_1:=U_1\setminus Z$. Then $V_1$ is open and dense in $U_1$. Furthermore, $\cX_{U_1} \to U_1$ has constant geometric genus over $V_1$ since so does $\cX_{U_1}' \to U'_1$ over $V_1'$, and the pullback of $V_1$ is a subset of $V_1'$. If the total space $\cX_{U_1}$ is reducible, we apply the above reasoning to each irreducible component of $\cX_{U_1}$ separately, and define $V_1$ to be the intersection of the dense open sets obtained in this way.

For an irreducible component $U_2$ of the complement $U_1 \setminus V_1$, the same procedure gives rise to a dense open subset $V_2 \subset U_2$ such that the restriction of $\cX_{U_2} \to U_2$ to $V_2$ has fibers of constant geometric genus. Repeating this process for all irreducible components of the $U_i \setminus V_j$, gives a finite stratification of $U_1$ by the locally closed $V_j$'s. By construction, each $V_{g,\CL}$ is a union of $V_j$'s, and hence a constructible subset of $|\CL|$ as asserted.

It remains to prove that the constructible subsets $\bigcup_{g'\le g}V_{g',\CL}$ and $\bigcup_{g'<g}V_{g',\CL}$ are closed in $U_1$, since then $V_{g,\CL}$ is locally closed. Let $R$ be a discrete valuation ring, and $\phi\:\Spec(R)\to U_1$ a map taking the generic point $\eta\in\Spec(R)$ to $[C_\eta]\in V_{g',\CL}$ for some $g'\le g$. After a quasi-finite base change, we may assume that the normalization of $C_\eta\to\eta$ is smooth and admits a stable model over $\Spec(R)$. Let $[C]\in U_1$ be the image of the closed point of $\Spec(R)$. Then $C$ is a reduced curve dominated by a semi-stable reduction of the normalization $C^\nu_\eta$. Thus, $p_g(C)\le p_g(C_\eta)=g'$. Therefore $\bigcup_{g'\le g}V_{g',\CL}$ and $\bigcup_{g'<g}V_{g',\CL}$ are closed in $U_1$ by the valuative criterion \cite[\href{https://stacks.math.columbia.edu/tag/0ARK}{Tag~0ARK}]{stacks-project}, and we are done.
\end{proof}

In general, Severi varieties may have components of different dimensions; see \cite{CC99}. However in the case of toric surfaces, they are equidimensional as the following proposition shows:	
	
\begin{prop}\label{prop:severi variety dimension}
If $S$ is a toric surface, then the Severi varieties $V^\irr_{g,\CL}$ and $V_{g,\CL}$ are either empty or equidimensional of dimension $-\CL.K_S+g-1$. Furthermore, a general $[C]\in V_{g,\CL}$ corresponds to a curve $C$ that intersects the boundary divisor transversally.
\end{prop}

\begin{proof}
Since $V^\irr_{g,\CL}$ is a union of irreducible components of $V_{g,\CL}$, it is sufficient to prove the assertion for $V_{g,\CL}$. Let $V$ be an irreducible component of $V_{g,\CL}$, and $[C] \in V$ general. Suppose $C$ has $m$ irreducible components $C_1,\dotsc, C_m$. Set $g_i:=\pg(C_i)$ and $\CL_i:=\CO_S(C_i)$, and consider the map $\prod_{i=1}^m V_{g_i, \CL_i}\to \overline{V}_{g,\CL}$. Its fiber over $[C]$ is finite, and the map is locally surjective. Thus, $\dim_{[C]}(V)=\sum_{i=1}^m\dim_{[C_i]}(V_{g_i, \CL_i})$. However, $\dim_{[C_i]}(V_{g_i, \CL_i})\le -\CL_i.K_S+g_i-1$ by \cite[Theorem~1.2]{Tyo13}, and hence
$$\dim_{[C]}(V)\le\sum_{i=1}^m\left(-\CL_i.K_S+g_i-1\right)=-\CL.K_S+g-1.$$		
Furthermore, if $C$ intersects the boundary divisor non-transversally, then the inequality is strict. Next, let us show the opposite inequality, which will imply both assertions of the proposition.

Let $X$ be the normalization of $C$, and $f\:X\to S$  the corresponding map. It defines a point $[X, f]$ in the Kontsevich moduli space $\CM_{g,n}\left (S, |\CL|\right )$ of stable maps to $S$ with image in $|\CL|$; see \cite{Kon95} and \cite[Theorem~5.7]{JHS}. Since $f$ is birational onto its image, it is automorphism-free, and hence $\CM_{g,n}\left (S, |\CL|\right )$ is a quasi--projective scheme in a neighborhood of $[X, f]$. Let $(R,\fm)$ be the formal completion of the algebra of functions on $\CM_{g,n}\left (S, |\CL|\right )$ at $[X, f]$. By \cite[Th\'eor\`eme III.2.1.7]{Ill71}, the tangent space to $\CM_{g,n}\left (S, |\CL|\right )$ at $[X, f]$ is given by $\left(\fm/\fm^2\right)^*=H^0(X, \mathcal N_{f})$, where $\mathcal N_{f}=\mathrm{Coker}(T_X\to f^*T_{S})$ is the normal sheaf to $f$, and the obstruction space is a subspace in $H^1(X, \mathcal N_{f})$.

Let us express the algebra $(R,\fm)$ as a quotient of a ring of formal power series $(A,\fn)$ by an ideal $I\subset \fn$ such that $\fn/\fn^2\simeq\fm/\fm^2$. Then, by a standard argument in deformation theory, the ideal $I$ is generated by at most $h^1(X, \mathcal N_{f})$ elements, and hence the dimension of any local germ of $\CM_{g,n}\left (S, |\CL|\right )$ at $[X,f]$ is bounded from below by $\chi(X, \CN_{f})$; cf. \cite[Proposition~3]{Mori79}. Since $c_1(\CN_f)=-\CL.K_{S}+2g-2$, it follows from the Riemann-Roch theorem that
$$\dim_{[X,f]}\CM_{g,n}\left (S, |\CL|\right )\ge \chi(X, \CN_{f}) = c_1(\CN_{f}) + 1 - g=-\CL.K_S+g-1.$$
		
Consider the projection from a small neighborhood of $[X,f]$ to $|\mathcal L|$ given by sending a point $[X',f']$ to $[f'(X')]$. Since the fiber over $[C]$ is finite, it follows that
$$\dim_{[C]}(V)\ge \dim_{[X,f]}\CM_{g,n}\left (S, |\CL|\right )\ge -\CL.K_S+g-1,$$ and we are done.
\end{proof}

\begin{rem}\label{rem:dimvgd}
If $M$ is the lattice of monomials of $S$, and $(S,\CL)$ is the polarized toric surface corresponding to a convex lattice polygon $\Delta\subset M_\RR$, then $-\CL.K_S=|\partial \Delta \cap M|$, and hence the dimension of the Severi varieties can be expressed combinatorially in terms of $\Delta$ and $g$. In particular, $\dim(V_{g,d})=3d+g-1$.
\end{rem}

\section{Tropical curves} \label{sec: tropical curves}
In this section, we review the theory of (parameterized) tropical curves, and recall the notions of floor decomposed curves and moduli spaces of parameterized tropical curves. We mainly follow \cite{Mik05, GM07, BM08, Tyo12, ACP}, to which we refer for further details. We also discuss families of parameterized tropical curves, cf. \cite{CCUW,Ran19}, and introduce the notions of harmonicity and local combinatorial surjectivity of the induced map to the moduli space. Throughout the section, $M$ and $N$ denote a pair of dual lattices.

\subsection{Families and parameter spaces for tropical curves}
\subsubsection{Abstract tropical curves} The graphs we consider have half-edges, called {\em legs}. A {\em tropical curve} is a weighted metric graph $\Gamma=(\mathbb G,\ell)$ with ordered legs, i.e., $\mathbb G$ is a (connected) graph with ordered legs equipped with a {\em weight (or genus) function} $g\colon V(\mathbb G)\rightarrow \mathbb Z_{\geq 0}$, and a {\em length function} $\ell\colon E(\mathbb G)\rightarrow \mathbb R_{>0}$. Here $V(\GG)$ and $E(\GG)$ denote the set of vertices and edges of $\GG$, respectively. We denote by $L(\GG)$ the set of legs of $\GG$ and extend $\ell$ to legs by setting their length to be infinite. Set $\overline E(\mathbb G) := E(\GG) \cup L(\GG)$. We will view tropical curves as polyhedral complexes by identifying the edges of $\GG$ with bounded closed intervals of corresponding lengths in $\RR$ and identifying the legs of $\GG$ with semi-bounded closed intervals in $\RR$.
	
For $e\in \overline E(\mathbb G)$ we denote by $e^\circ$ the interior of $e$, and use $\vec e$ to indicate a choice of orientation on $e$. If $e\in L(\GG)$ is a leg, then it will {\em always} be oriented away from the vertex. Bounded edges will be considered with both possible orientations. For $v\in V(\mathbb G)$, we denote by $\Star(v)$ the {\em star} of $v$, i.e., the collection of oriented edges and legs having $v$ as their tail. In particular, $\Star(v)$ contains two edges for every loop based at $v$. The number of edges and legs in $\Star(v)$ is called the {\em valency} of $v$, and is denoted by $\valence(v)$. By abuse of notation, we will often identify $\Star(v)$ with its geometric realization -- an oriented tree with root $v$, all of whose edges are leaves. 

The {\em genus} of $\Gamma$ is defined to be $g(\Gamma) = g(\GG) :=1-\chi(\mathbb G)+\sum_{v\in V(\mathbb G)}g(v)$, where $\chi(\GG):=b_0(\GG)-b_1(\GG)$ denotes the Euler characteristic of $\GG$. Finally, a tropical curve $\Gamma = (\GG, \ell)$ is called {\em stable} if so is $\GG$, that is, if every vertex of weight zero has valency at least three, and every vertex of weight one has valency at least one.
	
\subsubsection{Parameterized tropical curves}

A {\em parameterized tropical curve} is a pair $(\Gamma,h)$, where $\Gamma=(\mathbb G,\ell)$ is a tropical curve, and $h\colon \Gamma\rightarrow N_\mathbb R$ is a map such that:
	
(a) for any $e\in \overline E(\mathbb G)$, the restriction $h|_e$ is an integral affine function; and	
	
(b) (Balancing condition) for any vertex $v\in V(\mathbb G)$, we have $\sum_{\vec e\in\Star(v)}\frac{\partial h}{\partial \vec e}=0.$
	
\noindent Note that the slope $\frac{\partial h}{\partial \vec e}\in N$ is not necessarily primitive, and its integral length is the stretching factor of the affine map $h|_e$. We call it the {\it multiplicity} of $h$ along $\vec e$. In particular, $\frac{\partial h}{\partial \vec e} = 0$ if and only if $h$ contracts $e$. A parameterized tropical curve $(\Gamma, h)$ is called {\em stable} if so is $\Gamma$. Its {\em combinatorial type} $\Theta$ is defined to be the weighted underlying graph $\mathbb G$ with ordered legs equipped with the collection of slopes $\frac{\partial h}{\partial \vec e}$ for $e \in \overline E(\GG)$. We define the {\em extended degree} $\overline\nabla$ to be the sequence of slopes $\left(\frac{\partial h}{\partial \vec l}\right)_{l\in L(\GG)}$, and the {\em degree} $\nabla$ to be the subsequence $\left(\frac{\partial h}{\partial \vec l_i}\right)$, where $l_i$'s are the legs not contracted by $h$. The (extended) degree clearly depends only on the combinatorial type $\Theta$. An {\em isomorphism} of parameterized tropical curves $h\: \Gamma \to N_\RR$ and $h'\: \Gamma' \to N_\RR$ is an isomorphism of metric graphs $\varphi\: \Gamma \to \Gamma'$ such that $h = h' \circ \varphi$. Similarly, an {\em isomorphism} of combinatorial types $\Theta$ and $\Theta'$ is an isomorphism of the underlying graphs respecting the order of the legs, the weight function, and the slopes of the edges and legs. We denote the group of automorphisms of a combinatorial type $\Theta$ by $\mathrm{Aut}(\Theta)$, and the isomorphism class of $\Theta$ by $[\Theta]$.

\subsubsection{Families of tropical curves}\label{subsubsec:families}
Next we define families of (parameterized) tropical curves; cf. \cite[\S 3.1]{CCUW} and \cite[\S 2.2]{Ran19}. We will adapt these more general definitions to the special case we need -- families of (parameterized) tropical curves over a tropical curve $\Lambda$ without stacky structure. Since tropical curves are nothing but parameterized tropical curves with $N=0$, we discuss only families of parameterized tropical curves.

Let $\Lambda=(\mathbb G_\Lambda,\ell_\Lambda)$ be a tropical curve \textit{without loops}. Loosely speaking, a family of parameterized tropical curves over $\Lambda$ is a continuous family such that for any $e\in E(\GG_\Lambda)$, the combinatorial type of the fiber is constant along $e^\circ$, and the lengths of the edges of the fibers, as well as the images of the vertices of the fibers in $N_\RR$, form integral affine functions on $e$. To define this notion formally, consider a datum (\dag) consisting of the following:

\begin{itemize}
\item an extended degree $\overline\nabla$;
\item a combinatorial type $\Theta_e=\left(\GG_e,(\frac{\partial h}{\partial\vec\gamma})\right)$ of extended degree $\overline\nabla$ for each $e \in \overline E(\mathbb G_\Lambda)$;
\item an integral affine function $\ell(\gamma,\cdot)\:e\to \RR_{\ge 0}$ for each $e \in \overline E(\mathbb G_\Lambda)$ and $\gamma\in E(\GG_e)$;
\item an integral affine function $h(u,\cdot):e\to N_\RR$ for each $e \in \overline E(\mathbb G_\Lambda)$ and $u\in V(\GG_e)$;
\item a parameterized tropical curve $h_w\:\Gamma_w\to N_\RR$ of extended degree $\overline\nabla$ for each $w \in V(\mathbb G_\Lambda)$;
\item a weighted contraction $\varphi_{\vec{e}}\: \mathbb{G}_e \to \mathbb{G}_w$ preserving the order of the legs for each $w\in V(\mathbb G_\Lambda)$ and $\vec{e} \in \Star(w)$.
\end{itemize}
For any $e \in \overline E(\mathbb G_\Lambda)$ and $q\in e^\circ$, set $\Gamma_q:=(\mathbb{G}_e, \ell_q)$, where $\ell_q\:E(\GG_e)\to \RR_{\ge 0}$ is the function defined by $\ell_q(\gamma):=\ell(\gamma,q)$.

\begin{defn}\label{def:familyofparamtrcur}
Let $\Lambda$ be a tropical curve without loops. We say that a datum (\dag) is a {\em family of parameterized tropical curves over $\Lambda$} if the following compatibilities hold for all $w \in V(\mathbb G_\Lambda)$, $\vec{e} \in \Star(w)$, and $q\in e^\circ$:
\begin{enumerate}
\item $(\Gamma_q,h_q)$ is a parameterized tropical curve of combinatorial type $\Theta_e$, where $h_q$ is the unique piecewise affine map for which $h_q(u)=h(u,q)$  for all $u\in V(\GG_e)$, and $\left(\frac{\partial h_q}{\partial \vec l}\right)=\overline\nabla$;
\item  the length of $\varphi_{\vec{e}}(\gamma)$ in $\Gamma_w$ is $\ell(\gamma,w)$ for all $\gamma\in E(\GG_e)$;
\item $h_w(\varphi_{\vec{e}}(u))=h(u,w)$ for all $u\in V(\GG_e)$.
\end{enumerate}
A family of parameterized tropical curves over $\Lambda$ will be denoted by $h\:\Gamma_\Lambda\to N_\RR$.
\end{defn}
	
\subsubsection{The parameter space}
We denote by $M_{g, n, \nabla}^\trop$ the parameter space of isomorphism classes of genus $g$ stable parameterized tropical curves of degree $\nabla$ with exactly $n$ contracted legs. We assume that the contracted legs are $l_1,\dotsc,l_n$. Similarly to the case of the parameter space of abstract tropical curves studied in \cite[Section 2]{ACP}, we consider $M_{g, n, \nabla}^\trop$ as a {\em  generalized polyhedral complex}, i.e. $M_{g, n, \nabla}^\trop$ is glued from ``orbifold'' quotients of polyhedra $\overline M_\Theta$ factored by certain finite groups of automorphisms $\mathrm{Aut}(\Theta)$. In particular, the integral affine structure on the quotient is determined by that of $M_\Theta$, and the ``new'' faces of the geometric quotient $\overline M_\Theta/\mathrm{Aut}(\Theta)$ are not considered to be faces of the orbifold quotient; see {\em loc.cit.} for more details. More explicitly, $M_{g, n, \nabla}^\trop$ is constructed as follows.

A parameterized tropical curve of type $\Theta$ defines naturally a point in $N_\mathbb R^{|V(\mathbb G)|}\times \mathbb R_{>0}^{|E(\mathbb G)|}$, and the set of parameterized
tropical curves of type $\Theta$ gets identified with the interior $M_\Theta$ of a convex polyhedron $\overline M_\Theta\subseteq N_\mathbb R^{|V(\mathbb G)|}\times \mathbb R_{\ge 0}^{|E(\mathbb G)|}$; see, e.g., \cite[Proposition 2.13]{Mik05} or \cite[\S~3]{GM07} for details. The lattice $N^{|V(\mathbb G)|}\times \ZZ^{|E(\mathbb G)|}\subset N_\mathbb R^{|V(\mathbb G)|}\times \mathbb R^{|E(\mathbb G)|}$ defines an integral affine structure on $\overline M_\Theta$.

For each subset $E\subseteq E(\GG)$, the intersection of $\overline M_\Theta$ with the locus
$$
\{(n_u,x_e)\in N_\mathbb R^{|V(\mathbb G)|}\times \mathbb R_{\ge 0}^{|E(\mathbb G)|}\,|\, x_e=0\;\; \text{if and only if}\;\; e\in E \}$$
is either empty or can be identified naturally with $M_{\Theta_E}$, where $\Theta_E$ is the type of degree $\nabla$ and genus $g$, obtained from $\Theta$ by the weighted edge contraction of $E\subset E(\GG)$. Plainly, the corresponding inclusion $\iota_{\Theta,E}\:\overline M_{\Theta_E} \hookrightarrow \overline M_\Theta$ respects the integral affine structures.

Next, notice that an isomorphism $\alpha\:\Theta_1\to\Theta_2$ of combinatorial types induces an isomorphism $N_\mathbb R^{|V(\mathbb G_1)|}\times \mathbb R^{|E(\mathbb G_1)|}\to N_\mathbb R^{|V(\mathbb G_2)|}\times \mathbb R^{|E(\mathbb G_2)|}$ that takes $M_{\Theta_1}$ to $M_{\Theta_2}$ and also respects the integral affine structures. We thus obtain isomorphisms
$\iota_\alpha\:\overline M_{\Theta_1}\to \overline M_{\Theta_2}.$
In particular, the group $\mathrm{Aut}(\Theta)$ acts naturally on $\overline M_\Theta$.

The space $M_{g, n, \nabla}^\trop$ is defined to be the colimit of the diagram, whose entries are $\overline M_\Theta$'s for all combinatorial types $\Theta$ of genus $g$, degree $\nabla$ curves having exactly $n$ contracted legs $l_1,\dotsc,l_n$; and arrows are the inclusions $\iota_{\Theta,E}$'s and the isomorphisms $\iota_\alpha$'s described above. By the construction, for each type $\Theta$ of degree $\nabla$ and genus $g$ with exactly $n$ contracted legs $l_1,\dotsc,l_n$, we have a finite map $M_\Theta \to M_{g, n, \nabla}^\trop$, whose image $M_\Theta/\mathrm{Aut}(\Theta)$ is denoted by $M_{[\Theta]}$.

\begin{defn}
Let $\Lambda$ be a tropical curve, and $\alpha\:\Lambda\to M_{g, n, \nabla}^\trop$ a continuous map. We say that $\alpha$ is {\em piecewise integral affine} if for any $e\in\overline{E}(\Lambda)$ the restriction $\alpha_{|_e}$ lifts to an integral affine map $e \to \overline M_\Theta$ for some combinatorial type $\Theta$.
\end{defn}

\begin{prop}
Let $h\:\Gamma_\Lambda\to N_\RR$ be a family of parameterized tropical curves of degree $\nabla$. Then the induced map $\alpha\:\Lambda\to M_{g, n, \nabla}^\trop$ is piecewise integral affine.
\end{prop}
\begin{proof}
Follows immediately from the definitions.
\end{proof}

\begin{defn}\label{def:harloccombsurj}
Let $\Lambda$ be a tropical curve, $\alpha\:\Lambda\to M_{g, n, \nabla}^\trop$ a piecewise integral affine map, $w$ a vertex of $\Lambda$, and $\Theta$ a type such that $\alpha(w)\in M_{[\Theta]}$.
\begin{enumerate}
 \item Suppose $\alpha(\Star(w))\subset M_{[\Theta]}$. We say that $\alpha$ is {\em harmonic} at $w$ if $\alpha_{|_{\Star(w)}}$ lifts to a harmonic map $\alpha_w\:\Star(w)\to M_\Theta$, i.e., $\alpha_w$ is piecewise integral affine, and $\sum_{\vec{e}\in \Star(w)}\frac{\partial \alpha_w}{\partial \vec{e}}=0$.
 \item Suppose $\alpha(\Star(w))\nsubseteq  M_{[\Theta]}$. We say that $\alpha$ is \textit{locally combinatorially surjective} at $w$ if for any combinatorial type $\Theta'$ with an inclusion $\iota_{\Theta',E}\: \overline M_\Theta\hookrightarrow \overline M_{\Theta'}$, there exists a germ $\vec e$ of $\Star(w)$ such that $\alpha(\vec e)\cap M_{[\Theta']}\neq \emptyset $.
\end{enumerate}
\end{defn}

\subsubsection{Regularity}\label{sec:regularity}
A parameterized tropical curve $(\Gamma,h)$, its type $\Theta$, and the corresponding stratum $M_{[\Theta]}$ are called {\em regular} if $M_\Theta$ has the expected dimension, 
\begin{equation*} \label{eq:expected_dim}
    \mathrm{expdim}(M_\Theta) \coloneqq |\nabla|+n+(\rank(N)-3)\chi(\GG)-\ov(\GG),
\end{equation*}
 where $\GG$, as usual, denotes the underlying graph, $\chi(\GG)$ its Euler characteristic, and $\ov(\GG)$ the overvalency of $\GG$, i.e., $\ov(\GG):=\sum_{v\in V(\Gamma)}\max\{0,\valence(v)-3\}$. By \cite[Proposition~2.20]{Mik05}, the expected dimension provides a lower bound on the dimension of a stratum $M_\Theta$. Note that in \emph{loc. cit.}, the expected dimension is stated to be $\mathrm{expdim}(M_\Theta) - c$, where $c$ denotes the number of edges contracted by $h$. 
 This discrepancy in the formulae for the expected dimension appears because in \emph{loc. cit.} only deformations of the image $h(\Gamma)$ are taken into account. Remembering also $\Gamma$ gives one additional parameter for each contracted edge, and these parameters are independent. 
 
 The following is Mikhalkin's \cite[Proposition~2.23]{Mik05} stated in terms of the current paper:

\begin{prop}\label{prop:mikhbound}
Assume that $\rank(N)=2$ and $h$ is an immersion away from the contracted legs. If $\GG$ is weightless and $3$-valent, then $(\Gamma,h)$ is regular, and $\dim(M_\Theta)=|\nabla|+n+g-1$. On the other hand, if $\ov(\GG)>0$, then $\dim(M_\Theta)\le |\nabla|+n-\chi(\GG)-1\le |\nabla|+n+g-2$.
\end{prop}

\subsubsection{The evaluation map}\label{sec:evaluation}
Consider the natural evaluation map $\ev\: M_{g, n, \nabla}^\trop\to N_{\RR}^{n}$, defined by $\ev(\Gamma,h):=\left(h(l_1),\dotsc, h(l_n)\right)$. Let $\Theta$ be a combinatorial type, $(q_1,\dotsc,q_n)\in N_\RR^n$ any tuple, and $\ev_\Theta$ the composition $\overline{M}_\Theta\to M_{g, n, \nabla}^\trop\to N_{\RR}^{n}$. Notice that $\ev_\Theta$ is the restriction of a linear projection $N_\RR^{|V(\GG)|} \times \RR^{|E(\GG)|} \to N_\RR^n$. Thus, $\ev_\Theta^{-1}(q_1,\dotsc,q_n)$ is a polyhedron cut out in $\overline{M}_\Theta$ by an affine subspace, and hence its boundary is disjoint from $M_\Theta$ unless $\ev_\Theta^{-1}(q_1,\dotsc,q_n)$ is a point.

\begin{defn}
We say that a tuple $(q_1,\dotsc,q_n)\in N_\RR^n$ is {\em general with respect to} $\nabla$ and $g$ if for any combinatorial type $\Theta$ of degree $\nabla$ and genus $g$ with exactly $n$ contracted legs, either the codimension of $\ev_\Theta^{-1}(q_1,\dotsc, q_n)$ in $\overline{M}_\Theta$ is $n\cdot\rank(N)$, as expected, or $\ev_\Theta^{-1}(q_1,\dotsc, q_n)=\emptyset$.
\end{defn}

In what follows, tuples of points will be called general, without mentioning $\nabla$ and $g$, which we tacitly assume to be fixed. Since there are only finitely many isomorphism classes of types $\Theta$ of fixed degree and genus, the set of non-general tuples is a nowhere dense closed subset in $N_\RR^n$.
	
\begin{rem} \label{rem:mik_genpos} 
There are other ways to define tropical general position in the literature. In particular, we will use Mikhalkin's definition \cite[Definition 4.7]{Mik05} in the proof of the next proposition. For $\rank(N) = 2$ a tuple of points $(q_1, \dotsc, q_n) \in N_\RR^n$ is in general position in the sense of Mikhalkin if for any parameterized tropical curve $(\Gamma, h)$ of genus $g$ and with $n$ contracted legs $l_1, \dotsc, l_n$, such that $h(l_i) = q_i$ and with $n \geq |\nabla| + g - 1$ the following holds: 
\begin{enumerate}
    \item The underlying graph of $\Gamma$ is weightless and $3$-valent, all slopes are non-zero, except for the slopes of the contracted legs, and if $h(q) = h(q')$ for two points $q, q' \in \Gamma$, then neither $q$ nor $q'$ is a vertex and there is no third point of $\Gamma$ mapped to $h(q)$,
    \item the only vertex contained in $h^{-1}(q_i)$ is the one adjacent to $l_i$, and 
    \item $n = |\nabla| + g - 1$.
\end{enumerate}
\end{rem}
	
\begin{prop}\label{prop:genmikgen}
Assume that $\rank(N) = 2$, $n=|\nabla|+g-1$, and $(q_1,\dotsc,q_n)\in N_\RR^n$ is a general tuple. Let $\Theta$ be a combinatorial type of degree $\nabla$ and genus $g$ with $n$ contracted legs, and let $\GG$ be its underlying weighted graph. If $\ev_\Theta^{-1}(q_1,\dotsc, q_n)\ne\emptyset$, then $\GG$ is $3$-valent and weightless, and all slopes are non-zero, except for the slopes of the contracted legs.
\end{prop}

\begin{proof}
Let $\xi\in\ev_\Theta^{-1}(q_1,\dotsc, q_n)$ be any point. Since weightless $3$-valent graphs are not contractions of other stable weighted graphs, we may assume that $\xi\in M_\Theta$. The map $\ev_\Theta$ is the restriction of a linear projection $N_\RR^{|V(\GG)|} \times \RR^{|E(\GG)|} \to N_\RR^n$ to the polyhedral cone $\overline M_\Theta\subset N_\RR^{|V(\GG)|} \times \RR^{|E(\GG)|}$. By definition of points in general position, $\ev_\Theta^{-1}(q_1,\dotsc, q_n)$ has codimension $2n$ in $\overline M_\Theta$. Since $N_\RR^n$ has dimension $2n$, it follows that $\ev_\Theta$ is a submersion at any point in the open cone $M_\Theta$, in particular at the point $\xi$. By \cite[Corollary~4.12]{Mik05}, the tuples of points in tropically general position in the sense of Mikhalkin are dense, and hence there exists a small perturbation $\xi'\in \ev_\Theta^{-1}(q'_1,\dotsc, q'_n)$ of $\xi$, where $(q'_1,\dotsc, q'_n)$ is a tuple in tropically general position in the sense of Mikhalkin. Thus the assertion holds by definition (see Remark~\ref{rem:mik_genpos}).
\end{proof}

\begin{defn}\label{defn:simple wall}
A stratum $M_{[\Theta]}\subset M^\trop_{g,n,\nabla}$ is called {\em nice} if it is regular, and the underlying graph $\GG$ is weightless and $3$-valent. A stratum $M_{[\Theta]}$ is called {\em a simple wall} if it is regular, and $\GG$ is weightless and $3$-valent except for a unique $4$-valent vertex.
\end{defn}
	
\begin{lem}\label{lem:dimbound}
Assume that $(q_1, \dotsc, q_n) \in N_\RR^n$ is general, $\ev_\Theta^{-1}(q_1,\dotsc, q_n)\ne\emptyset$, and $\rank(N)=2$. Then,
\begin{enumerate}		
\item If $n=|\nabla|+g-1$, then $M_{[\Theta]}$ is nice;
		
\item If $n=|\nabla|+g-2$ and $M_{[\Theta]}$  is a simple wall, then $\ev_\Theta^{-1}(q_1,\dotsc, q_n)$ is a point in $M_\Theta$;
		
\item If $n=|\nabla|+g-2$ and $M_{[\Theta]}$ is nice, then $\ev_\Theta^{-1}(q_1,\dotsc, q_n)\subset\overline M_\Theta$ is an interval, whose boundary is disjoint from $M_\Theta$.
\end{enumerate}
\end{lem}
	
\begin{proof}
Assertion (1) follows from Proposition~\ref{prop:genmikgen}. Thus, we assume that $n=|\nabla|+g-2$. If $M_{[\Theta]}$ is a simple wall, then it is regular, and the graph $\GG$ is weightless and has overvalency one. Therefore, $$\dim\left(\ev_\Theta^{-1}(q_1,\dotsc, q_n)\right)=|\nabla|+n+g-2-2n=0,$$ and hence $\ev_\Theta^{-1}(q_1,\dotsc, q_n)$ is a point in $M_\Theta$, as asserted in (2). Similarly, if $M_{[\Theta]}$ is nice, then the dimension count shows that $\ev_\Theta^{-1}(q_1,\dotsc, q_n)$ is an interval, which implies assertion (3).
\end{proof}

\subsection{Floor decomposed curves}\label{subsec:floordecomp}
Next we recall the notion of a floor decomposed curve as introduced by Brugall\'e and Mikhalkin \cite{BM08}. We will follow the presentation in \cite[\S 4.4]{BIMS}, see also \cite[\S 2.5]{rau17}. For the rest of this section, we assume $M=N=\mathbb Z^2$ and $\Delta\subset M_\mathbb R=\mathbb R^2$ is an $h$-transverse polygon (cf. \cite[\S 2]{BM08}), e.g., the triangle $\Delta_d$ with vertices $(0,0)$, $(0,d)$ and $(d,0)$. We associate to $\Delta$ the reduced degree $\nabla$ of tropical curves dual to $\Delta$, i.e., $\nabla$ consists of primitive outer normals to the sides of $\Delta$, and the number of slopes outer normal to a given side is equal to its integral length.

\tikzset{every picture/.style={line width=0.75pt}} 
\begin{figure}[ht]
\begin{tikzpicture}[x=0.5pt,y=0.5pt,yscale=-0.7,xscale=0.7]
\import{./}{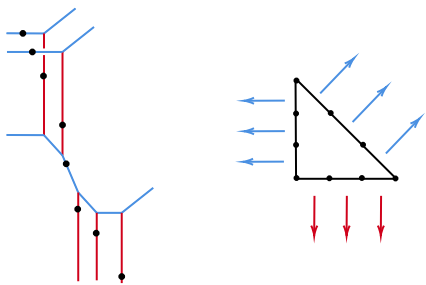}
\end{tikzpicture}
\caption{A collection of vertically stretched points, and a floor decomposed cubic tropical curve passing through them and having the reduced tropical degree associated to $\Delta_3$.}
\label{fig:floor}
\end{figure}

\begin{defn}\label{defn:vertical stretched point}
	Let $\lambda\in\RR$ be a positive real number. A point configuration $q_1,\dotsc,q_n$ in $\mathbb R^2$ is called \textit{vertically $\lambda$-stretched} if for any pair of distinct points $q_i=(x,y)$ and $q_j=(x',y')$ in the configuration the following holds: $|y-y'|>\lambda \cdot |x-x'|$.
\end{defn}

\begin{defn}\label{defn:floor decomposed curve}
	A parameterized tropical curve $(\Gamma,h)$ is called \textit{floor decomposed} if for any edge $e$ of $\Gamma$, the slope $\frac{\partial h}{\partial \vec{e}}$ is either $(\pm 1,*)$ or $(0,*)$.
\end{defn}

Let $(\Gamma,h)$ be a floor decomposed parameterized tropical curve. If the image of a non-contracted edge (resp. leg) is vertical, then the edge (resp. leg) is called an \textit{elevator}. After removing the interiors of all elevators in $\Gamma$, we are left with a disconnected graph. The non-contracted connected components of this graph are called the \textit{floors} of $\Gamma$. In Figure~\ref{fig:floor}, the elevators are the red edges, and the blue components are the floors.

\begin{prop}\label{prop:floor decomposed curves}\cite[\S 5.1]{BM08}
	Let $g$ be an integer, and set $n := |\nabla|+g-1$. Then there exists $\lambda$ such that for any generic configuration $q_1,\dotsc, q_n\subset \mathbb R^2$ of vertically $\lambda$-stretched points, every parameterized tropical curve $(\Gamma, h)$ in $\ev^{-1}(q_1,\dotsc, q_n)\subset M^\trop_{g,n,\nabla}$ is floor decomposed. Furthermore, each floor and each elevator of $\Gamma$ contains exactly one of the $q_i$'s in its image.
\end{prop}

By abuse of language, in what follows we will call a collection of points as in Proposition~\ref{prop:floor decomposed curves} just {\em vertically stretched}.

\begin{rem}\label{rem:multiplicity of elevator}
If $\Delta=\Delta_d$, then there is a unique elevator adjacent to the top floor, and, by the balancing condition, for every floor there is a downward elevator adjacent to it. The elevator adjacent to the top floor and the downward elevators adjacent to the bottom floor -- all have multiplicity one. If in addition the points $q_1,\dotsc,q_n$ belong to the line given by $y=-\mu x$ for $\mu\gg 1$, then the $x$-coordinate of the marked point belonging to the image of the top floor is smaller than the $x$-coordinate of any other $q_i$.
\end{rem}

\begin{rem}\label{rem:elevators}
	In the original definition of \cite{BM08}, an elevator either connects two floors or is an infinite ray attached to a floor. In our convention, an elevator of \cite{BM08} may split into a couple of elevators separated by a vertex mapped to one of the $q_i$'s.
\end{rem}

\section{Tropicalization}\label{sec:tropicalizatoin of families of curves}

The goal of this section is to construct the tropicalization of certain one-parameter families of algebraic curves with a map to a toric variety (Theorem~\ref{thm:paramtrop}). As a result we obtain a tropical tool (Corollary~\ref{cor:tropicalization_inf}) for studying algebraic degenerations that plays a central role in the proof of the Main Theorem.

\subsection{Notation and terminology}\label{subsec:notation in tropicalization}
Throughout the section we assume that $K$ is the algebraic closure of a complete discretely valued field $F$. We denote the valuation by $\val\:K\to \RR\cup\{\infty\}$, the ring of integers by $K^0$, and its maximal ideal by $K^{00}$. We fix a pair of dual lattices $M$ and $N$, and a toric variety $S$ with lattice of monomials $M$. In this paper we will mostly be interested in the case $S=\PP^2$, but the tropicalization we are going to describe works in general.

By a {\em family of curves} we mean a {\em flat, projective} morphism of finite presentation and relative dimension one. By a collection of {\em marked points} on a family of curves we mean a collection of disjoint sections contained in the smooth locus of the family. A family of curves with marked points is {\em prestable} if its fibers have at-worst-nodal singularities; cf. \cite[\href{https://stacks.math.columbia.edu/tag/0E6T}{Tag~0E6T}]{stacks-project}. It is called {\em (semi-)stable} if so are its geometric fibers. A prestable curve with marked points over a field is called {\em split} if the irreducible components of its normalization are geometrically irreducible and smooth, and the preimages of the nodes in the normalization are defined over the ground field. A family of prestable curves with marked points is called {\em split} if all of its fibers are so; cf. \cite[\S~2.22]{dJ96}. If $U\subset Z$ is open, and $(Y,\sigma_\bullet)$ is a family of curves with marked points over $U$, then by a {\em model} of $(Y,\sigma_\bullet)$ over $Z$ we mean a family of curves with marked points over $Z$, whose restriction to $U$ is $(Y,\sigma_\bullet)$.

\subsection{Tropicalization for curves}\label{subsec:tropicalization for curves}

We first recall the canonical tropicalization construction for a fixed (parameterized) curve over $K$. This is well established, and we refer to \cite{Tyo12} and \cite{BPR} for details. The normalization of signs is chosen such that the algebraic definition is compatible with the standard tropical pictures. A \emph{parameterized curve} in $S$ is a smooth projective curve with marked points $(X,\sigma_\bullet)$ and a map $f\: X \to S$ such that $f(X)$ does not intersect orbits of codimension greater than one, and the image of $X \setminus \left( \bigcup_i \sigma_i \right)$ under $f$ is contained in the dense torus $T \subset  S$.

Let $f\: X \to S$ be a parameterized curve, and $X^0 \to \Spec(K^0)$ a prestable model. Denote by $\widetilde{X}$ the fiber of $X^0$ over the closed point of $\Spec(K^0)$. As usual, a point $s \in \widetilde X$ is called \textit{special} if it is either a node or a marked point of $\widetilde X$. Let $D = \Sigma_j D_j$ be the boundary divisor of $S$. Set $f^*(D_j):= \Sigma_r d_{j, r} p_{j, r}$ with $d_{j, r} \in \NN$ and $p_{j, r} \in X$. Plainly, $\{p_{j,r}\}\subseteq \{\sigma_i\}$ by definition. The collection of multisets $\{d_{j,r}\}_r$ for each $j$ is called the \emph{tangency profile} of $f\: X \to S$. We say that the tangency profile is trivial if $d_{j,r} = 1$ for all $j$ and $r$.

The \emph{tropicalization} $\trop(X)$ of $X$ with respect to the model $X^0$ is the tropical curve $\Gamma=(\mathbb G,\ell)$ defined as follows: The underlying graph $\GG$ is the dual graph of the central fiber $\widetilde X$, i.e., the vertices of $\mathbb G$ correspond to irreducible components of $\widetilde{X}$, the edges -- to nodes, the legs -- to marked points, and the natural incidence relation holds. For a vertex $v$ of $\mathbb G$, its weight is defined to be the geometric genus of the corresponding component of the reduction $\widetilde{X}_v$. As for the length function, if $e\in E(\GG)$ is the edge corresponding to a node $z\in \widetilde{X}$, then $\ell(e)$ is defined to be the valuation of $\lambda$, where $\lambda\in K^{00}$ is such that \'etale locally at $z$, the total space of $X^0$ is given by $xy=\lambda$. Although $\lambda$ depends on the \'etale neighborhood, its valuation does not, and hence the length function is well-defined. Finally, notice that the order on the set of marked points induces an order on the set of legs of $\Gamma$. By abuse of notation, we will not distinguish between the tropical curve $\trop(X)$ and its geometric realization.

Next, we explain how to construct the parameterization $h\:\trop(X)\to N_\RR$. Let $\widetilde{X}_v$ be an irreducible component of $\widetilde X$. Then, for any $m \in M$, the pullback $f^*(x^m)$ of the monomial $x^m$ is a non-zero rational function on $X^0$, since the preimage of the big orbit is dense in $X$. Thus, there is $\lambda_m \in K^\times$, unique up to an element invertible in $K^0$, such that $\lambda_mf^*(x^m)$ is an invertible function at the generic point of $\widetilde X_v$. The function $h(v)$, associating to $m \in M$ the valuation $\val(\lambda_m)$, is clearly linear, and hence $h(v)\in N_\RR$. The parameterization $h\:\trop(X)\to N_\RR$ is defined to be the unique piecewise integral affine function with values $h(v)$ at the vertices of $\trop(X)$, whose slopes along the legs satisfy the following: for any leg $l$ and $m\in M$ we have $\frac{\partial h}{\partial \vec l}(m)=-\mathrm{ord}_{\sigma_i} f^*(x^m)$, where $\sigma_i$ is the marked point corresponding to $l$. Then $h\:\trop(X)\to N_\RR$ is a parameterized tropical curve, by \cite[Lemma 2.23]{Tyo12}. The curve $\trop(X)$ (resp. $h\:\trop(X)\to N_\RR$) is called the \emph{tropicalization} of $X$ (resp. $f\: X \to S$) with respect to the model $X^0$. Plainly, the tropical curve $\trop(X)$ is independent of the parameterization, and depends only on $X^0$. If the family $X\to S$ is stable and $X^0$ is the stable model, then the corresponding tropicalization is called simply {\em the tropicalization} of $X$ (resp. $f\: X \to S$). Plainly, $\trop(X)$ (resp. $h\:\trop(X)\to N_\RR$) is stable in this case.

\begin{rem}\label{rem:slopesoflegs}
Let $l$ be the leg corresponding to a marked point $\sigma_i$. If $\sigma_i$ is mapped to the boundary divisor $D$, then, by definition, $f(\sigma_i)$ belongs to a unique component of $D$, and in particular, to the smooth locus of $D$. Thus, the slope $\frac{\partial h}{\partial \vec l}$ is completely determined by the irreducible component of $D$ containing the image of $\sigma_i$. Furthermore, the multiplicity of $\frac{\partial h}{\partial \vec l}$ is the multiplicity of $\sigma_i$ in $f^*D$.
Note also that $\frac{\partial h}{\partial \vec l} = 0$ if and only if $f(\sigma_i)$ is contained in the dense orbit $T\subset S$.
\end{rem}

Next, let us recall the {\em tropicalization map} $\trop\: X(K)\backslash \bigcup_i \sigma_i \to \trop(X)$ from $K$-points of the curve to the tropicalization of the curve. The image $\trop(\eta)$ of a $K$-point $\eta\in X(K)\backslash \bigcup_i \sigma_i$ is defined as follows. We temporarily add $\eta$ as a marked point to obtain the curve $(X,\sigma_\bullet,\eta)$ with marked points $\{\sigma_i\}_i \cup \eta$. We consider the minimal modification $(X^0)'\to X^0$ such that $(X^0)'$ is a prestable model of $(X,\sigma_\bullet,\eta)$. Then $\trop(X)'$ is obtained from $\trop(X)$ by attaching a leg either to an existing vertex or to a new two-valent vertex splitting an edge or a leg of $\trop(X)$. Either way, the geometric realization of $\trop(X)'$ is obtained from $\trop(X)$ by attaching a leg to some point $q\in\trop(X)$. The tropicalization map $\trop\: X(K)\backslash\bigcup_i \sigma_i \to \trop(X)$ sends $\eta$ to the point $q$. 

\begin{rem}\label{rem:pointwise tropicalization}
The tropicalization defined above is compatible with the natural tropicalization of the torus in the sense that the following diagram is commutative:
$$
\begin{tikzcd}
    X\backslash f^{-1}(D)=f^{-1}(T)\rar{\mathrm{trop}}\dar{f}&\mathrm{trop}(X)\dar{h}\\ T\rar{\mathrm{trop}} &N_\mathbb R
\end{tikzcd}
$$
where the map $\trop\:T\to N_\RR$ is defined by $p\mapsto \left(m\mapsto -\val(x^m(p))\right)$. In particular, in the notation of Remark~\ref{rem:slopesoflegs}, if $h$ contracts $l$, then $h(l)=\trop(f(\sigma_i))$.
\end{rem}

Below we will need a more explicit description of $\trop(\eta)\in\trop(X)$.
Let $\psi\: \Spec(K) \to X$ be the immersion of the point $\eta$. Since $X^0 \to \Spec(K^0)$ is proper, $\psi$ admits a unique extension $\psi^0\:\Spec(K^0) \to  X^0$ with image $\overline \eta$. Let $s\in\Spec(K^0)$ be the closed point. Its image in $\widetilde X$ under the map $\psi^0$ is called the {\em reduction} of $\eta$, and is denoted by $\red(\eta)$. If the reduction of $\eta$ is a non-special point of a component $\widetilde X_v$, then $(X^0)'= X^0$ and $\trop(\eta)=v$. If the reduction of $\eta$ is a node $z\in\widetilde X$, then $\trop(\eta)$ belongs to the edge $e$ corresponding to $z$, and we shall specify the distance from $\trop(\eta)$ to the two vertices $v$ and $w$ adjacent to $e$. Let $\lambda \in K^{00}$ be such that $X^0$ is given by $ab=\lambda$ \'etale locally at $z$, and let $\widetilde X_v, \widetilde X_w$ be the branches of $\widetilde X$ associated to the vertices $v$ and $w$, respectively. Without loss of generality, $\widetilde X_v$ is given locally by $a=0$ and $\widetilde X_w$ by $b=0$.

\begin{lem}\label{lem:dist_in_edge}
The distance from $\trop(\eta)$ to $w$ in $e$ is given by $\val(\psi^*(a))$.
\end{lem}

\begin{proof}
The model $(X^0)'$ is a blow up of $X^0$ at the point $z$, and its local charts are given by $at=\mu$ and $bt^{-1}=\lambda\mu^{-1}$. In the first chart, the exceptional divisor is given by $a=0$ and $\widetilde X_w$  by $t=0$. Furthermore, $\psi^*(t)\in K^0$ is invertible since $\eta$ specializes to a non-special point of the exceptional divisor. Thus, $\val(\psi^*(a))=\val(\mu)$ is the distance from $\trop(\eta)$ to $w$ in $e$.
\end{proof}

Finally, if the reduction of $\eta$ is the specialization of a marked point $\sigma_i$, and $l$ is the leg corresponding to $\sigma_i$, then $\trop(\eta)$ belongs to the leg $l$. To specify the distance from $\trop(\eta)$ to the vertex $w$ adjacent to $l$, let $a=0$ be an \'etale local equation of $\sigma_i^0$ in $X^0$ around $\psi^0(s)$. Then,
\begin{lem}\label{lem:dist_in_leg}
The distance from $\trop(\eta)$ to $w$ in $l$ is given by $\val(\psi^*(a))$.
\end{lem}

\begin{proof}
The proof is completely analogous to that of Lemma~\ref{lem:dist_in_edge}.
\end{proof}

\subsection{Tropicalization for one-parameter families of curves}\label{subsec:topicalization of families}
In this section, we describe the tropicalization procedure for one-dimensional families of stable curves with a map to a toric variety. We expect that the main existence statement -- the first part of Theorem~\ref{thm:paramtrop} -- is known to experts; see, in particular, \cite{ACGS}, \cite{CCUW} and \cite{Ran19} for related constructions. We use the local (toroidal) structure of the moduli space of pointed curves and the universal curve over it; cf. \cite[Theorem~5.2]{DM69} and \cite[Theorem~2.7]{Knu83}.
\begin{defn}\label{def: family of paramcur}
A \emph{family of parameterized curves in $S$} consists of the following data:
\begin{enumerate}
    \item a smooth, projective base curve with marked points $(B,\tau_\bullet)$,
    \item a family of stable marked curves $(\cX \to B, \sigma_\bullet)$, smooth over $B':=B \setminus \left( \bigcup_i \tau_i \right)$, and
    \item a rational map  $f \colon \CX \dashrightarrow S$, defined over $B'$, such that for any closed point $b \in B'$, the restriction $f_b\colon\cX_b \to S$ is a parameterized curve.
\end{enumerate}
\end{defn}

\begin{thm}\label{thm:paramtrop}
Let $f \colon \CX \dashrightarrow S$ be a family of parameterized curves, $B^0$ a prestable model of the base curve $(B,\tau_\bullet)$, and assume that the family $(\cX \to B, \sigma_\bullet)$ admits a split stable model $(\cX^0 \to B^0,\sigma_\bullet ^0)$. Let $\Lambda:=\trop(B)$ be the tropicalization of $(B,\tau_\bullet)$ with respect to $B^0$ and assume that $\Lambda$ has no loops. Then there exists a family of parameterized tropical curves $h\: \Gamma_\Lambda\to N_\RR$ such that for any $K$-point $\eta \in B'=B \setminus \left( \bigcup_i \tau_i \right)$, the fiber of $(\Gamma_\Lambda, h)$ over $\trop(\eta)$ is the tropicalization of $f\:\cX_\eta\to S$ with respect to the model $\cX^0|_{\overline \eta}$. Furthermore, the induced map $\alpha\:\Lambda\to M^{\trop}_{g,n,\nabla}$ is either harmonic or locally combinatorially surjective at any vertex $w\in V(\Lambda)$ for which the curve $\Gamma_w$ is weightless and $3$-valent except for at most one $4$-valent vertex.
\end{thm}

\begin{defn}
The family $h\:\Gamma_\Lambda \to N_\RR$ constructed in Theorem~\ref{thm:paramtrop} is called the \emph{tropicalization} of  $f \colon \CX \dashrightarrow S$ with respect to $\cX^0 \to B^0$.
\end{defn}

\begin{rem} \label{rem:mon_abstract}
Without the splitness assumption, the dual graphs $\GG_w$ for $w \in V(\Lambda)$ can be defined only up-to an automorphism. In general, one needs to consider families $\Gamma_\Lambda \to \Lambda$ with a {\em stacky structure}, following \cite{CCUW}.
\end{rem}

\begin{cor}\label{cor:tropicalization_inf}
Let $f\:\cX\dashrightarrow S$ be a family of parameterized curves, and $h\: \Gamma_\Lambda \to N_\RR$ its tropicalization with respect to $\cX^0 \to B^0$. Let $\tau\in B(K)$ be a marked point, $l\in L(\Lambda)$ the associated leg, and $\GG_l$ the underlying graph of the tropical curves parameterized by $l$. Assume that the lengths of all but one edge of $\GG_l$ are constant in the family $\Gamma_\Lambda$ and the map $h$ is constant on all vertices of $\GG_l$. Then,
\begin{enumerate}
	\item $\cX_\tau$ has exactly one node. In particular, the geometric genus of $\cX_\tau$ is one less than the geometric genus of a general fiber of $\cX\to B$;
	\item The rational map $f\: \cX \dashrightarrow S$ is defined on the fiber $\cX_\tau$. Furthermore, $f$ maps the generic point(s) of $\cX_\tau$ to the dense orbit $T\subset S$, and $f\:\cX_\tau\to S$ has the same tangency profile as the general fiber $f\:\cX_b\to S$.
\end{enumerate}
\end{cor}

\begin{rem}
Under the assumptions of the corollary, the edge of varying length gets necessarily contracted by the map $h$.
\end{rem}

The rest of the section is devoted to the proof of Theorem~\ref{thm:paramtrop} and Corollary~\ref{cor:tropicalization_inf}. Since $K$ is the algebraic closure of a complete discretely valued field $F$ as in Section~\ref{subsec:notation in tropicalization}, any $K$-scheme of finite type is defined over a finite extension $F'$ of $F$, which is also a complete discretely valued field since so is $F$. Thus, we may view a $K$-scheme of finite type as the base change of a scheme over a finite extension of $F$. In the proofs below, we will work over $F'$ in order to have a well behaved total space of the families we consider. To ease the notation, we will assume that all models and points we are interested in are defined already over $F^0$, but will not assume that the valuation of the uniformizer $\pi$ is one.

In particular, we may assume that $B^0,\cX^0, \cX^0\to B^0$, and $\overline \eta$ are all defined over $F^0$. In the proof, we will use the following notation: $\psi\: \Spec(F) \to B$ will denote the immersion of the point $\eta$, $s:=\red(\eta)\in \widetilde B$ its reduction, and $\widetilde{\cX_s}$ the corresponding fiber. We set $D_{B^0}:=\widetilde B \cup \left(\bigcup_i \tau_i\right)$ and $D_{\cX^0}:=\widetilde \cX \cup \left ( \bigcup_i \cX_{\tau_i} \right) \cup \left( \bigcup_j \sigma_j \right)$, where $\cX_{\tau_i}$ are the fibers over the marked points $\tau_i \in B(F)$. Then the pullback of any monomial function $f^*(x^m)$ is regular and invertible on $\cX^0 \setminus D_{\cX^0}$ by the assumptions of the theorem.

\subsection{Proof of Theorem~\ref{thm:paramtrop}}

\subsubsection*{Step 1: The tropicalization $h_\eta\:\trop(\cX_\eta)\to N_\RR$ depends only on $\trop(\eta)\in \Lambda$.}\label{step 1}
Set $q:=\trop(\eta)$. After modifying $B^0$, we may assume that $s\in \widetilde B$ is non-special and, by definition, this modification depends only on $\trop(\eta)$, cf. Section~\ref{subsec:tropicalization for curves}. Let $\widetilde B_w$ be the component containing $s$, and $w=q$ the corresponding vertex of $\Lambda$. By definition, the underlying graph of $\trop(\cX_\eta)$ is the dual graph $\GG_s$ of $\widetilde{\cX}_s$. Let $z$ be a node of $\widetilde{\cX}_s$. Then there exist an \'etale neighborhood $U$ of $s$ in $B^0$ and a function $g_z\in \CO_{U}(U)$ vanishing at $s$ such that the family $\cX^0\times_{B^0}U$ is given by $xy=g_z$  \'etale locally near $z$. After shrinking $U$, we may assume that the latter is true in a neighborhood of any node of $\widetilde{\cX}_s$ over $U$. Furthermore, since $s\in \widetilde B$ is non-special, we may choose $U$ such that the pullback of $D_{B^0}$ to $U$ is $\widetilde B_w$.

By assumption, $\cX$ is smooth over $B'$, and hence each $g_z$ is invertible on the complement of $D_{B^0}$, i.e., away from $\widetilde B_w$. However, $U$ is normal, and $g_z(s)=0$ for all nodes $z$. Thus, all the $g_z$'s vanish identically along $\widetilde B_w$, which implies that the dual graph $\GG_{s}$ of $\widetilde{\cX}_s$ is \'etale locally constant over $\widetilde{B}_w$. We claim that the lengths of the edges of $\GG_{s}$ are also \'etale locally constant. Indeed, pick a node $z\in \widetilde{\cX}_s$. Since $g_z$ is invertible away from $\widetilde B_w$, and $\pi$ vanishes to order one along $\widetilde B_w$, there exists $k_w \in \NN$ such that  $\pi^{-k_w}g_z $ is regular and invertible in codimension one, and hence, by normality, it is regular and invertible on $U$. Thus, the length of the edge of $\GG_s$ corresponding to $z$ is given by $\val(\psi^* g_z)=k_w\val(\pi)$, cf. Section~\ref{subsec:tropicalization for curves}, which is \'etale locally constant around $s$. Finally, since $\widetilde \cX|_{\widetilde{B}_w} \to \widetilde{B}_w$ is split, the identifications of graphs $\GG_{s}$ with the dual graph $\GG_w$ of the generic fiber is canonical. Hence the tropicalization $\trop(\cX_\eta)$ depends only on $q$, and we set $\Gamma_q:=\trop(\cX_\eta)$.

It remains to check that the parameterization $h_\eta$ also depends only on $q$. Since the tangency profile of $\cX_\eta\to S$ is independent of $\eta$, so are the slopes of the legs of $\trop(\cX_\eta)$, cf. Remark~\ref{rem:slopesoflegs}. Let $u$ be a vertex of $\GG_w$, and $\widetilde \cX_u\subset\widetilde \cX$ the component corresponding to $u$ that dominates $\widetilde B_w$. Pick any $m\in M$. Since $f^*(x^m)$ is regular and invertible away from $D_{\cX^0}$, $\cX^0$ is normal, and $\pi$ vanishes to order one along $\widetilde\cX_u$, it follows that there is an integer $k_{u}\in\ZZ$ such that $\pi^{k_{u}}f^*(x^m)$ is regular and invertible at the generic point of $\widetilde \cX_u$. Thus, $h_\eta(u)(m)=k_{u}\val(\pi)$, and hence $h_\eta$ depends only on $q$. We set $h_q:=h_\eta$, and obtain a parameterized tropical curve $h_q\:\Gamma_q\to N_\RR$ that is canonically isomorphic to the tropicalization of $f\:\cX_\eta\to S$ for any $K$-point $\eta\in B'$ satisfying $\trop(\eta)=q$.

\subsubsection*{Step 2: $h\:\Gamma_\Lambda\to N_\RR$ is a family of parameterized tropical curves.}\label{step 2} We need to show that the fiberwise tropicalizations constructed in Step 1 form a family of parameterized tropical curves; that is, we need to specify a datum (\dag) as in \S~\ref{subsubsec:families} that satisfies the conditions of Definition~\ref{def:familyofparamtrcur}. Since the tropicalizations of $K$-points give rise only to rational points in the tropical curve, we will work with rational points $\Lambda_\QQ\subset \Lambda$, and in the very end extend the family by linearity to the non-rational points of $\Lambda$.

By Step 1, the extended degree of $h_q\:\Gamma_q\to N_\RR$ is independent of $q\in \Lambda$, and will be denoted by $\overline \nabla$. Furthermore, if $q$ is an inner point of some $e \in \overline E(\mathbb G_\Lambda)$, and $\trop(\eta)=q$, then $s$ is the special point of $\widetilde{B}$ corresponding to $e$, and the underlying graph of $\Gamma_q$ is the dual graph of the reduction $\widetilde{\cX}_s$. Therefore, the underlying graph of $\Gamma_q$ depends only on $e$, and will be denoted by $\GG_e$. We will see below that in this case, the slopes of the bounded edges of $\Gamma_q$, also depend only on $e$. Hence so does the combinatorial type of $(\Gamma_q,h_q)$.

Next, we specify the contraction maps. Let $s'\in \widetilde B_w$ be a node of $\widetilde B$ corresponding to an edge $e$ of $\Lambda$, and $\GG_e:=\GG_{s'}$ the corresponding weighted graph. Any degeneration of stable curves corresponds to a weighted edge contraction on the level of dual graphs. Thus, for the \'etale local branch of $\widetilde B_w$ at $s'$ corresponding to $\vec{e} \in \Star(w)$, we get the contraction $\varphi_{\vec{e}}\: \GG_{e} \to \GG_w$ between the associated dual graphs. Analogously, for the reduction $s'\in \widetilde B_w$ of a marked point with corresponding leg $l\in L(\Lambda)$, we obtain the desired contraction $\varphi_{\vec l}\: \GG_{l} \to \GG_w$.

Finally, let us define the functions $\ell$ and $h$. Let $w \in V(\Lambda)$ be a vertex, and $\vec e\in \Star(w)$ an edge or a leg. We define the functions $\ell(\gamma,\cdot)\:e\cap\QQ\to \RR_{\ge 0}$ and $h(u,\cdot):e\cap\QQ\to N_\RR$ for each $\gamma\in E(\GG_e)$ and $u\in V(\GG_e)$ as follows: if $q \in e^{\circ}\cap\QQ$, then $\ell(\gamma, q):=\ell_q(\gamma)$ and $h(u, q):=h_q(u)$, and if $q$ is the tail of $\vec e$, then $\ell(\gamma, q):=\ell_q(\varphi_{\vec e}(\gamma))$ and $h(u, q):= h_q(\varphi_{\vec e}(u))$.

To finish the proof of Step 2, it remains to show that the functions $\ell(\gamma,\cdot)\:e\cap\QQ\to \RR_{\ge 0}$ and $h(u,\cdot)\:e\cap\QQ\to N_\RR$ are restrictions of integral affine functions. Indeed, if this is the case, then the family extends to the irrational points of $e$. Furthermore, it follows that the slope of $h_q$ along any $\gamma\in \overline{E}(\GG_e)$ is continuous on $e^\circ$, and obtains values in $N$ at the rational points $q\in e^\circ\cap\QQ$. Therefore, it is necessarily constant on $e^\circ$, and hence so is the combinatorial type.

Notice that for the function $h$, the assertion can be verified separately for the evaluations $h(u, \cdot)(m)$ of $h(u, \cdot)$ at single monomials $m$. Thus, we fix an arbitrary $m\in M$ for the rest of the proof. We also fix an edge $\gamma \in E(\GG_q)$ and a vertex $u\in V(\GG_q)$, and let $z \in \widetilde \cX_s$ denote the node corresponding to $\gamma$ and $\widetilde \cX_{s, u}$ the irreducible component of $\widetilde \cX_s$ corresponding to $u$; see Figure~\ref{fig:degen2}. There are two cases to consider.

\tikzset{every picture/.style={line width=0.75pt}} 
\begin{figure}[h]
\begin{tikzpicture}[x=0.5pt,y=0.5pt,yscale=-.8,xscale=.8]
\import{./}{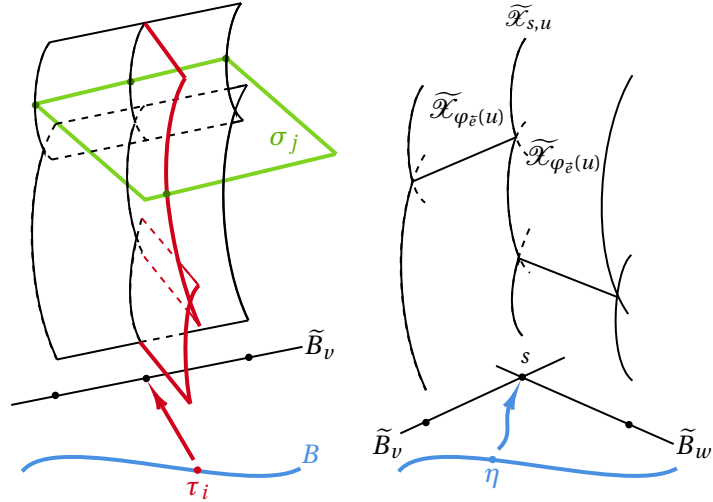}
\end{tikzpicture}		
\caption{The left picture shows the components of $D_{\cX^0}$. On the right, we indicate the irreducible components of $\widetilde \cX$.}\label{fig:degen2}
\end{figure}

Case 1: {\em $s\in\widetilde B$ is a node.} Let $e$ be the edge of $\Lambda$ corresponding to $s$. Then $q\in e^\circ$ and $\GG_q=\GG_e$. After shrinking $U$, we may assume that it is given by $ab=\pi^{k_s}$ for some positive integer $k_s$, and $\widetilde{B}$ has two components $\widetilde{B}_v$ and $\widetilde{B}_w$ in $U$ given by $a=0$ and $b=0$, respectively. Thus, the length $\ell_\Lambda(e)$ is given by $k_s\val(\pi)$.

We start with the function $\ell$. Since $g_z$ is regular and vanishes only along $\widetilde{B}_v \cup \widetilde{B}_w$ in $U$, there exist $k_a,k_b,k_\pi\in \NN$ such that $a^{-k_a}b^{-k_b}\pi^{-k_\pi}g_z$ is regular and invertible on $U$. Therefore,
\begin{equation}\label{eq:edgelength} \ell( \gamma, q)=\val(\psi^*g_z)=k_a\val(\psi^*a)+k_b\val(\psi^*b)+k_\pi\val(\pi)=(k_a-k_b)\val(\psi^*a)+k_b\ell_\Lambda(e)+k_\pi\val(\pi).
\end{equation}
By Lemma~\ref{lem:dist_in_edge}, $\val(\psi^*(a))$ is the distance of $q$ from $w$ in $e$. Thus, \eqref{eq:edgelength} defines an integral affine function of $q$ on $e$ with slope $k_a-k_b$; see Figure~\ref{fig:trop1} for an illustration.

\tikzset{every picture/.style={line width=0.75pt}} 
\begin{figure}[ht]
\begin{tikzpicture}[x=0.5pt,y=0.5pt,yscale=-0.8,xscale=0.8]
\import{./}{figtrop1.tex}
\end{tikzpicture}		
\caption{The tropicalization of $\cX \to B$ near a node of $\widetilde B$.
}\label{fig:trop1}
\end{figure}

It remains to show that the value of \eqref{eq:edgelength} for $\val(\psi^*(a)) = 0$ is $\ell_w(\varphi_{\vec e}(\gamma))$, where the orientation on $e$ is such that $w$ is its tail. By Step 1, $\ell_w(\varphi_{\vec e}(\gamma))=k_w\val(\pi)$. Since $a$ is invertible at the generic point of $\widetilde{B}_w$, the order of vanishing $k_w$ of $g_z$ at the generic point of $\widetilde{B}_w$ is equal to the order of vanishing of $b^{k_b}\pi^{k_\pi}$, which in turn is the order of vanishing of $\pi^{k_bk_s+k_\pi}$. Thus,
\begin{equation}\label{eq:slope}
\ell_w(\varphi_{\vec e}(\gamma))=k_w\val(\pi)=(k_bk_s+k_\pi)\val(\pi)=k_b\ell_\Lambda(e)+k_\pi\val(\pi),
\end{equation}
as needed.

Next we consider the function $h$. For the two orientations $\vec e$ and $\cev e$ on $e$, let $\varphi_{\vec e}\:\GG_e \to \GG_w$ and $\varphi_{\cev e}\:\GG_e \to \GG_v$ be the weighted edge contractions defined above. As in Step 1, we have irreducible components $\widetilde \cX_{\varphi_{\cev e}(u)}$ and $\widetilde \cX_{\varphi_{\vec e}(u)}$ of $\widetilde \cX$; they are the components of $\widetilde \cX$ containing $\widetilde \cX_{s,u}$ and supported over $\widetilde B_v$ and $\widetilde B_w$, respectively; see Figure~\ref{fig:degen2}. Since $f^*(x^m)$ is regular and invertible outside $D_{\cX^0}$, there exist $r_a,r_b,r_\pi\in\ZZ$ such that $a^{r_a}b^{r_b}\pi^{r_\pi}f^*(x^m)$ is regular and invertible at the generic points of $\widetilde \cX_{\varphi_{\cev e}(u)}$ and $\widetilde \cX_{\varphi_{\vec e}(u)}$, and hence, by normality of $\cX$, it is regular and invertible in a neighborhood of the generic point of $\widetilde \cX_{s,u}$. Therefore,

\begin{equation} \label{eq:interpolation_edge_map}
	h(u, q)(m)= r_a \val(\psi^*(a)) + r_b \val(\psi^*(b)) + r_\pi \val(\pi) = (r_a - r_b) \val(\psi^*(a)) + r_b \ell_\Lambda(e) + r_\pi \val(\pi).
\end{equation}
	By Lemma~\ref{lem:dist_in_edge}, $\val(\psi^*(a))$ is the distance of $q$ from $w$ in $e$. Thus, \eqref{eq:interpolation_edge_map} defines an integral affine function on $e$ with slope $r_a-r_b$. Furthermore, the value of this function for $\val(\psi^*(a)) = 0$ is $h_w(\varphi_{\vec e}(u))$. Indeed, since $a$ is regular and invertible at the generic point of $\widetilde B_w$, it is regular and invertible at the generic point of $\widetilde{\cX}_{\varphi_{\vec e}(u)}$. Thus, by the definition of $h_w$,
$$h_w(\varphi_{\vec e}(u))=\val\left(b^{r_b}\pi^{r_\pi}\right)=\val\left(\pi^{k_sr_b+r_\pi}\right)=(k_sr_b+r_\pi)\val(\pi)=r_b\ell_\Lambda(e)+r_\pi\val(\pi),$$
as needed. Similarly, the value of \eqref{eq:interpolation_edge_map} for $\val(\psi^*(a)) = \ell_\Lambda(e)$ is $h_v(\varphi_{\cev e}(u))$.

\vspace{5px}

Case 2: {\em $s\in \widetilde B$ is the reduction of a marked point}. Let $\tau\in B$ be the marked point with reduction $s$, $l$ the associated leg of $\Lambda$, and $\widetilde B_w$ the component of $\widetilde B$ containing $s$. Then $q\in l^\circ$ and $\GG_q=\GG_l$. After shrinking $U$, we may assume that $\widetilde B \cup \tau$ in $U$ is given by $\pi a=0$.

Again, we start with the function $\ell$. Since $g_z$ is regular in $U$ and vanishes only along $\widetilde B \cup \tau$, there exist $k_w, k_\tau\in \NN$ such that $\pi^{- k_w}a^{-k_{\tau}}g_z$ is regular and invertible on $U$. The length $\ell(\gamma, q)=\val(\psi^*g_z)$ is thus given by
\begin{equation} \label{eq:interpolation_leg}
	\ell(\gamma, q) = k_w \val(\pi) + k_{\tau} \val(\psi^*(a)).
\end{equation}
By Lemma~\ref{lem:dist_in_leg}, $\val(\psi^*(a))$ is the distance of $q$ from $w$ in $l$. Thus, \eqref{eq:interpolation_leg} defines an integral affine function on $l$ with slope $k_{\tau}$ and, by Step 1, its value at  $\val(\psi^*(a)) = 0$ is $k_w\val(\pi)=\ell_w(\varphi_{\vec l}(\gamma))$, as required; see Figure~\ref{fig:trop2} for an illustration.

\tikzset{every picture/.style={line width=0.75pt}} 
\begin{figure}[ht]
\begin{tikzpicture}[x=0.5pt,y=0.5pt,yscale=-.8,xscale=.8]
\import{./}{figtrop2.tex}
\end{tikzpicture}		
\caption{The tropicalization of $\cX \to B$ near a marked point of $\widetilde B$.}\label{fig:trop2}
\end{figure}

We proceed with the function $h$. Let $\varphi_{\vec l}\: \GG_l \to \GG_w$ be the edge contraction defined above. As before, we have a unique component $\widetilde \cX_{\varphi_{\vec l}(u)}$ of $\widetilde \cX$ containing $\widetilde \cX_{s,u}$ and supported over $\widetilde B_w$. Similarly, there is a unique component $\cX_{\tau,u}$ of $\cX_\tau$ containing $\widetilde \cX_{s,u}$ in its closure. Since $f^*(x^m)$ is regular and invertible away from $D_{\cX^0}$, there exist $r_w, r_\tau\in \ZZ$ such that $\pi^{r_w}a^{r_{\tau}}f^*(x^m)$ is regular and invertible on $U$. Therefore,
	\begin{equation} \label{eq:parameterized curves along leg}
	h(u, q)(m) = r_{\tau} \val(\psi^*(a)) + r_w \val(\pi).
	\end{equation}
By Lemma~\ref{lem:dist_in_leg}, $\val(\psi^*(a))$ is the distance of $q$ from $w$ in $l$. Thus, \eqref{eq:parameterized curves along leg} defines an integral affine function on $l$ with slope $r_{\tau}$, and the value $h_w(\varphi_{\vec l}(u))(m)$ at $w$, since $w$ is given by $\val(\psi^*(a)) = 0$, cf. Step 1.

\subsubsection*{Step 3: The harmonicity and the local combinatorial surjectivity of the map $\alpha$.}  Let $w\in V(\Lambda)$ be a vertex, $\Theta$ the combinatorial type of $\alpha(w)$, and $\mathcal C\to \overline{\mathcal M}_{g,n+|\nabla|}$ the universal curve. Consider the natural map to the coarse moduli space $\chi\:B^0\to \overline{M}_{g,n+|\nabla|}$. There are two cases to consider:

\vspace{5px}

Case 1: {\em $\chi$ contracts $\widetilde B_w$.} Set $p:=\chi(\widetilde B_w)$. Since $\widetilde \cX \to \widetilde B$ is split, it follows that the restriction of $\widetilde{\cX}$ to $\widetilde{B}_w$ is the product $\widetilde{B}_w\times \mathcal C_p$. Furthermore, $\GG_w$ is the dual graph of $\mathcal C_p$, and $\alpha$ maps $\Star(w)$ to $M_{[\Theta]}$. Since $\alpha$ is piecewise integral affine, the map lifts to a map to $M_\Theta$, which we denote by $\alpha_w$. Let us show that $\alpha$ is harmonic at $w$, i.e.,
\[
\sum_{\vec{e}\in \Star(w)}\frac{\partial \alpha_w}{\partial \vec{e}}=0.
\]
The latter equality can be verified  coordinatewise. Recall that the integral affine structure on $M_\Theta$ is induced from $N_\mathbb R^{|V(\GG_w)|}\times \mathbb R^{|E(\GG_w)|}$. Let $\gamma\in E(\GG_w)$ be an edge corresponding to a node $z\in\mathcal C_p$, and assume for simplicity that $\gamma$ is not a loop. The case of a loop can be treated similarly, and we leave it to the reader. Let $u,u'\in V(\GG_w)$ be the vertices adjacent to $\gamma$, $\mathcal C_u$ and $\mathcal C_{u'}$ the corresponding components of $\mathcal C_p$, and $\widetilde{\cX}_u, \widetilde{\cX}_{u'}, Z$ the pullbacks of $\mathcal C_u, \mathcal C_{u'}, z$ to $\widetilde{\cX}$, respectively. The universal curve $\mathcal C$ is given \'etale locally at $z$ by $xy=m_z$, where $m_z$ is defined on an \'etale neighborhood of $p$, and vanishes at $p$. Thus, $\cX^0$ is given by $xy=g_Z:=\chi^*m_z$ in an \'etale neighborhood of $Z$. Notice that we constructed a function  defined in a neighborhood of the whole family of nodes $Z$, which globalizes the local construction of Step 1 in the particular case we consider here.

To prove harmonicity with respect to the coordinate $x_{\gamma}$ corresponding to $\gamma$, we now consider $(\pi^{-k_w}g_Z)|_{\widetilde B_w}$, where, as in Step 1, $k_w$ denotes the order of vanishing of $g_Z$ at the generic point of $\widetilde B_w$. Thus,  $(\pi^{-k_w}g_Z)|_{\widetilde B_w}$ is a non-zero rational function on $\widetilde B_w$, and hence the sum of orders of zeroes and poles of this function is zero. We claim that this is precisely the harmonicity condition we are looking for. Indeed, since $\cX^0$ is smooth over $B'$, the function $g_Z$ is regular and invertible away from $D_{B^0}$, and hence, as usual, $(\pi^{-k_w}g_Z)|_{\widetilde B_w}$ has zeroes and poles only at the special points of $\widetilde{B}_w$. For $\vec{e}\in\Star(w)$, let $s$ be the corresponding special point. If $s$ is a node of $\widetilde{B}$, and $\widetilde{B}_v$ is the second irreducible component containing $s$, then pick $k_s\in\NN$ as in Case 1 of Step 2, i.e., such that $B^0$ is given \'etale locally at $s$ by $ab=\pi^{k_s}$. Then, using Case 1 of Step 2,
\[
\frac{\partial x_{\gamma}}{\partial \vec{e}}=\frac{\ell_v(\gamma)-\ell_w(\gamma)}{\ell_\Lambda(e)}=\frac{k_v-k_w}{k_s},
\]
which is the order of vanishing of $(\pi^{-k_w}g_Z)|_{\widetilde B_w}$ at $s$. Similarly, if $s$ is the specialization of a marked point, $\frac{\partial x_{\gamma}}{\partial \vec{e}}$ is again the order of vanishing of $(\pi^{-k_w}g_Z)|_{\widetilde B_w}$ at $s$.
	
Next, let us show harmonicity with respect to a coordinate $n_u$ for $u\in V(\GG_w)$, that is, we need to show that
\[
\sum_{\vec{e}\in \Star(w)}\frac{\partial h(u, \cdot)(m)}{\partial \vec{e}}=0
\]
for any $m\in M$. Let $k_u\in\ZZ$ be such that $\pi^{k_u}f^*(x^m)$ is regular and invertible at the generic point of $\widetilde \cX_{u}$. Then the divisor $D_{u,m}$ of $(\pi^{k_u}f^*(x^m))|_{\widetilde \cX_{u}}$ has horizontal components supported on the special points of the fibers of $\widetilde \cX\to \widetilde{B}$, and vertical components supported on the preimages of the special points of $\widetilde{B}_w$. Pick a general point $c\in \mathcal C_u$, and consider the horizontal curve $\widetilde{B}_{w,c}:=\widetilde{B}_w\times \{c\}\subset \widetilde \cX_{u}$. It intersects no horizontal components of $D_{u,m}$, and intersects its vertical components transversally.

Let $s\in \widetilde{B}_w$ be a special point, and $\widetilde{\cX}_{s,u}$ the corresponding vertical component of $D_{u,m}$. Its multiplicity in $D_{u,m}$ is equal to the multiplicity of the point $(s,c)$ in the divisor of $(\pi^{k_u}f^*(x^m))|_{\widetilde{B}_{w,q}}$, and hence the sum over all special points of $\widetilde{B}_w$ of these multiplicities vanishes. On the other hand, we claim that the multiplicity of $\widetilde{\cX}_{s,u}$ in $D_{u,m}$ is nothing but $\frac{\partial h(u, \cdot)(m)}{\partial \vec{e}}$, where $\vec{e}\in \Star(w)$ corresponds to the special point $s$. Indeed, if $s$ is the specialization of a marked point $\tau$, then $\frac{\partial h(u, \cdot)(m)}{\partial \vec{e}}=r_\tau$, where $r_\tau$ is the order of pole of $f^*(x^m)$ at $\tau$ as in Case 2 of Step 2. Since $\pi$ does not vanish at $\tau$, we conclude that $\frac{\partial h(u, \cdot)(m)}{\partial \vec{e}}$ is the multiplicity of the reduction $\widetilde{\cX}_{s,u}$ of $\cX_\tau$ in $D_{u,m}$, as asserted. A similar computation shows that if $s$ is a node of $\widetilde{B}$, then $\frac{\partial h(u, \cdot)(m)}{\partial \vec{e}}$ is again the multiplicity of $\widetilde{\cX}_{s,u}$ in $D_{u,m}$. We leave the details to the reader.


\tikzset{every picture/.style={line width=0.75pt}}
\begin{figure}[ht]
\begin{tikzpicture}[x=0.7pt,y=0.7pt,yscale=-.65,xscale=.65]
\import{./}{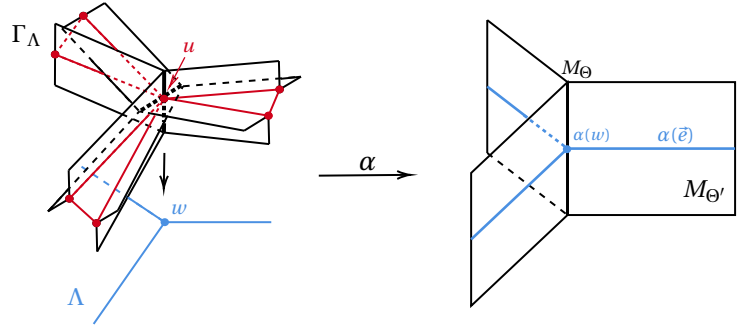}
\end{tikzpicture}		
\caption{An illustration for Case 2 of Step 3 in the proof of Theorem \ref{thm:paramtrop}. Since $\Gamma_\Lambda$ is not embeddable in $\RR^3$, the picture is a ``cartoon''. The valency of $w$ may be greater than three, but for each resolution of the $4$-valent vertex, there is at least one germ in $\Star(w)$ as in the picture.}\label{fig:alpha}
\end{figure}

Case 2: {\em $\chi$ does not contract $\widetilde B_w$.} By the assumptions of the theorem, the graph $\GG_w$ is weightless and $3$-valent except for at most one $4$-valent vertex. Since rational curves with three special points have no moduli, it follows that $\GG_w$ has a $4$-valent vertex, which we denote by $u\in\GG_w$. In this case we will show that the map $\alpha$ is locally combinatorially surjective at $w$. Consider $\widetilde \cX_u$ as above. By construction, $\widetilde \cX_u\to \widetilde{B}_w$ is a family of rational curves and since $\widetilde \cX \to \widetilde B$ is split, the four marked points in each fiber define sections of the family. Let $\xi\:\widetilde B_w \to \overline{M}_{0, 4}\simeq \PP^1$ be the induced map. Since $\chi$ does not contract $\widetilde B_w$, it follows that  $\xi$ is not constant, and hence surjective. We conclude that $\widetilde B_w$ contains points, the fibers over which have dual graphs corresponding to the three possible splittings of the $4$-valent vertex $u$ into a pair of $3$-valent vertices joined by an edge. In particular, $\alpha$ does not map $\Star(w)$ into $M_{[\Theta]}$, and we need to show that for each $\Theta'$ with an inclusion $M_\Theta \hookrightarrow \overline M_{\Theta'}$ there is $\vec e \in \Star(w)$ such that $\alpha(e) \cap M_{[\Theta']} \neq \emptyset$. To see this, notice that any such polyhedron $M_{\Theta'}$ corresponds to one of the three possible splittings of the $4$-valent vertex $u$ since, by balancing, the slope of the new edge is uniquely determined. \qed

\subsection{Proof of Corollary~\ref{cor:tropicalization_inf} }

Set $s':=\red(\tau) \in \widetilde B$, and let $z$ be the node of the fiber $\widetilde\cX_{s'}$ corresponding to the edge of varying length $\gamma \in E(\GG_l)$. Let $\widetilde B_w$ be the component of $\widetilde B$ containing $s'$. Recall that we set $D_{\cX^0}:=\widetilde \cX \cup \left ( \bigcup_i \cX_{\tau_i} \right) \cup \left( \bigcup_j \sigma_j \right)$, where $\sigma_j$ are the marked points of $\cX \to B$.
	
(1) As in Case 2 of Step 2 in the proof of Theorem~\ref{thm:paramtrop}, in an \'etale neighborhood $U$ of $s'$,  $\widetilde B \cup \tau$ is given by $\pi a=0$, and the family $\cX^0\to B^0$ over $U$ is given \'etale locally near $z$ by $xy=g_z$ for some $g_z\in\mathcal O_U(U)$. The length of $\gamma$ is an integral affine function on $l$, and by \eqref{eq:interpolation_leg}, its slope is given by the order of vanishing $k_\tau$ of $g_z$ at $\tau$. Since the slope of $\ell(\gamma, \cdot)$ along $l$ is not constant, $k_\tau>0$, i.e., $g_z$ vanishes at $\tau$. Thus, $\cX_\tau$ has a node as asserted. Vice versa, any node of $\cX_\tau$ specializes to some node $z'$ of $\cX_{s'}$ corresponding to an edge $\gamma' \in E(\GG_l)$. Then \'etale locally at $z'$, the family $\cX^0\to B^0$ is given by $xy=g_{z'}$, and $g_{z'}$ vanishes at $\tau$. By \eqref{eq:interpolation_leg} we conclude that the slope of $\ell(\gamma',\cdot)$ is not constant along $l$. Thus, $\gamma = \gamma'$ by the assumption of the corollary, and hence $z'=z$.
	
(2) Let $\cX'_\tau\subseteq \cX_\tau$ be an irreducible component. First, we show that for any $m\in M$, the pullback $f^*(x^m)$ is regular and invertible at the generic point of $\cX'_\tau$. Hence the rational map $f$ is defined at the generic point of $\cX'_\tau$, and maps it to the dense orbit $T$. Pick a vertex $u$ of $\GG_l$ such that $\widetilde \cX_{s,u}$ belongs to the closure of $\cX'_\tau$. By \eqref{eq:parameterized curves along leg}, the slope of $h(u, \cdot)(m)$ along $l$ is given by the order of pole of $f^*(x^m)$ along $\cX'_\tau$. However, $h(u, \cdot)(m)$ is constant along $l$ by the assumptions of the corollary, and hence $f^*(x^m)$ has neither zero nor pole at the generic point of $\cX'_\tau$. Second, notice that since $\cX^0$ is normal, $f^*(x^m)$ is regular and invertible away from $D_{\cX^0}$, and it is regular and invertible in codimension one on $\cX_\tau$, it follows that $f$ is defined on $\cX_\tau\setminus \left( \bigcup_j \sigma_j \right)$. Pick any $\sigma_j$, and let us show that $f$ is defined at $\sigma_j(\tau)\in \cX_\tau$, too. Let $S_j\subseteq S$ be the affine toric variety consisting of the dense torus orbit and the orbit of codimension at most one containing the image of $\sigma_j(b)$ for a general $b\in B$. Then the pullback of any regular monomial function $x^m\in \mathcal O_{S_j}(S_j)$ is regular in codimension one in a neighborhood of $\sigma_j(\tau)$, and hence regular at $\sigma_j(\tau)$ by normality of $\cX^0$. Thus, $f$ is defined at $\sigma_j(\tau)$, and maps it to $S_j$. The last assertion is clear. \qed

\section{Degeneration via point constraints}\label{sec:degeneration}
In this section, $(S,\CL)=\left(\PP^2,\CO_{\PP^2}(d)\right)$, $M=N=\mathbb Z^2$, and $\nabla=\nabla_d$ is the reduced degree of tropical curves associated to the triangle $\Delta_d$, i.e., $\nabla$ consists of $3d$ vectors: $(1,1), (-1,0), (0,-1)$, each appearing $d$ times. Recall from Lemma~\ref{lem:constr} and Proposition~\ref{prop:severi variety dimension} that for an integer $1-d\le g\le \binom{d-1}{2}$, the Severi variety parameterizing curves of degree $d$ and geometric genus $g$ is a locally closed subset $V_{g,d}\subseteq |\CO_{\PP^2}(d)|$ of pure dimension $3d+g-1$. The goal of this section is to prove our first Main Theorem and its corollary, which generalizes Zariski's theorem to arbitrary characteristic.

\begin{thm}\label{thm:cont_lines}
Let $d\in\NN$ and $1-d\le g\le \binom{d-1}{2}$ be integers, and let $V\subseteq V_{g,d}$ be an irreducible component. Then $\overline V$ contains $V_{1 - d,d}$.
\end{thm}

\begin{rem}
If the geometric genus of a degree-$d$ curve $C$ is $1-d$, then $C$ is necessarily a union of $d$ lines. Therefore, $V_{1 - d,d}$ is dominated by $\left((\PP^2)^*\right)^d$. In particular, $V_{1 - d,d}$ is irreducible and its general element corresponds to a nodal curve.
\end{rem}

\begin{cor}[Zariski's Theorem] \label{cor:zariski}
Let $V \subset V_{g, d}$ be an irreducible subvariety. Then,
\begin{enumerate}
\item $\dim(V)\le 3d+g-1$, and
\item if $\dim(V)=3d+g-1$, then for a general $[C]\in V$, the curve $C$ is nodal.
\end{enumerate}
\end{cor}
\begin{proof}
Assertion (1) follows from Proposition~\ref{prop:severi variety dimension}. If $\dim(V)=3d+g-1$, then $V$ is an irreducible component of $V_{g, d}$ by Proposition~\ref{prop:severi variety dimension}, and hence $V_{1 - d,d}\subseteq \overline V$  by Theorem~\ref{thm:cont_lines}. Thus, there exists $[C]\in \overline V$ such that $C$ is nodal. However, the locus of nodal curves is open in $\overline V$ by \cite[\href{https://stacks.math.columbia.edu/tag/0DSC}{Tag~0DSC}]{stacks-project}, and $V$ is irreducible. Hence for a general $[C]\in V$, the curve $C$ is nodal.
\end{proof}
\begin{rem}
Assertion (1) in arbitrary characteristic and for any toric surface was already proved in \cite[Theorem~1.2]{Tyo13} by the third author. Examples of toric surfaces for which assertion (2) fails can also be found {\em loc.cit.}
\end{rem}

The rest of the section is devoted to the proof of Theorem~\ref{thm:cont_lines}, which proceeds by induction on $(d,g)$ with the lexicographic order. The base of induction, $(d,g)=(1,0)$, is clear. To prove the induction step, let $(d,g)>(1,0)$ be a pair of integers such that $1-d\le g\le \binom{d-1}{2}$, and assume that for all $(d',g')<(d,g)$ the assertion is true. Let us prove that it holds true also for the pair $(d,g)$. If $g=1-d$, then there is nothing to prove since the variety $V_{1-d,d}$ is irreducible. Thus, we may assume that $g>1-d$.

\subsubsection*{Step 1: The reduction to the case of irreducible curves.}
We claim that it is enough to prove the assertion for $V\subseteq V^\irr_{g,d}$. Indeed, let $[C]\in V$ be a general point, and assume that $C$ is reducible. Denote the degrees of the components by $d_i$, and their geometric genera by $g_i$. Then $d_i<d$ for all $i$, $d=\sum d_i$, and $g-1=\sum (g_i-1)$. Consider the natural finite map
$\prod \overline{V}_{g_i, d_i}^{\irr} \to \overline V_{g,d}$,
whose image contains $[C]$. By Proposition \ref{prop:severi variety dimension}, the dimension of each irreducible component of
$\overline{V}_{g_i, d_i}^{\irr}$ is equal to $3d_i+g_i-1$, and since $\sum_{i=1}^m (3d_i+g_i-1)=3d+g-1=\dim(V)$, it follows that $\overline V$ is dominated by a product of irreducible components $\overline V_i\subseteq \overline V_{g_i, d_i}^{\irr}$. By the induction assumption, $V_{1-d_i,d_i}\subseteq \overline V_i$ for all $i$. But a general point of $\prod \overline{V}_{1-d_i, d_i}$ corresponds to a union of $d=\sum d_i$ general lines, and hence the map $\prod \overline{V}_{g_i, d_i}^{\irr} \to \overline V_{g,d}$ takes it to a general point of $\overline V_{1 - d,d}$. It follows now that $V_{1 - d,d}\subseteq \overline V$. From now on, we assume that $V\subseteq V^\irr_{g,d}$. \vspace{5px}

Our goal is to prove that the locus of curves of geometric genus $g-1$ in $\overline V$ has dimension $3d+g-2$, since then $\overline V$ necessarily contains a component of $V_{g-1,d}$ by Proposition~\ref{prop:severi variety dimension}, and hence also $V_{1-d,d}$ by the induction assumption. We proceed as follows: set $n:=3d+g-1$, and pick $n-1$ points $\{p_i\}_{i=1}^{n-1}\subset\PP^2$ in general position. For each $i$, let $H_i\subset|\mathcal O_{\PP^2}(d)|$ be the hyperplane parameterizing curves passing through the point $p_i$. Since $\dim(V)=n$ and the points $\{p_i\}_{i=1}^{n-1}\subset\PP^2$ are in general position, the intersection $Z:=V\cap \left(\bigcap_{i=1}^{n-1} H_i\right)$ has pure dimension one and a general $[C]\in Z$ corresponds to an integral curve $C$ of geometric genus $g$ that intersects the boundary divisor transversally by Proposition~\ref{prop:severi variety dimension}. It is sufficient to show that there exists $[C']\in \overline Z$ such that $C'$ is reduced and has geometric genus $g-1$. Indeed, such a $C'$ passes through a general collection of $n-1$ points $\{p_i\}_{i=1}^{n-1}\subset\PP^2$, and therefore the locus of curves of geometric genus $g-1$ in $\overline V$ has dimension $n-1=3d+g-2$. 

In order to apply the results of the previous section, we assume that the field $K$ is the algebraic closure of a complete discretely valued field, which we may do by the Lefschetz principle. Furthermore, we may assume that the points $\{p_i\}_{i=1}^{n-1}\subset\PP^2$ tropicalize to distinct vertically stretched points $\{q_i\}_{i=1}^{n-1}$ in $\RR^2$. This assumption allows us to work on the tropical side with floor decomposed curves, which are very convenient for controlling the degenerations of curves parameterized by $Z$. The rest of the proof proceeds as follows. In Step 2, we modify the base curve $Z$ and the family of curves over it, so that both $Z$ and the general fiber of the family become smooth, and hence tropicalization results of Section~\ref{subsec:topicalization of families} apply. And in Step 3, we investigate the tropicalization of the modified family and prove that it necessarily contains a tropical curve that corresponds to an algebraic fiber of genus $g-1$. 

\subsubsection*{Step 2: The construction of a family of parameterized curves.}
This step follows the ideas and the techniques of de Jong \cite{dJ96} based on the results of Deligne \cite{D85}. The goal is to construct a family $f \colon \CX \dashrightarrow \PP^2$ of parameterized curves over a smooth base curve $(B,\tau_\bullet)$, a prestable model $B^0$ of $(B,\tau_\bullet)$ whose reduction $\widetilde B$ has smooth irreducible components, and a finite morphism $B\to Z$ that satisfy the following properties: (i) $(\cX \to B, \sigma_\bullet)$ extends to a split family of stable marked curves over $B^0$, and (ii) for a general $[C]\in Z$, the fibers of $\cX \to B$ over the preimages of $[C]$ are the normalization $X$ of $C$ equipped with the natural map to $\PP^2$ and with $3d+n-1$ marked points such that the first $n-1$ of them are mapped to $p_1,\dotsc,p_{n-1}$, and the rest -- to the boundary divisor; cf. Definition~\ref{def: family of paramcur}.

Let us start with the normalization $B\to Z$ and with the pullback $\CX\to B$ of the tautological family to $B$ equipped with the natural map $f\:\cX\to \PP^2$. We are going to replace $B$ with finite coverings, dense open subsets, and compactifications several times, but to simplify the presentation, we will use the same notation $B,\cX,$ and $f$. First, we apply Lemma~\ref{lma:normalization}. After replacing $B$ with $B'$ and $\cX$ with $\cX'$ as in the lemma, we may assume that the family $\cX\to B$ is a generically equinormalizable family of projective curves. After shrinking $B$, we may further assume that $\cX\to B$ is equinormalizable, the pullback of the boundary divisor of $\PP^2$ on each fiber is reduced, and the preimages of $p_i$'s are smooth points of the fibers.

Second, we label the points of $\CX$ that are mapped to $\{p_i\}_{i=1}^{n-1}$ and to the boundary divisor, which results in a finite covering of the base curve $B$. After replacing $B$ with this covering, and $\cX$ with the normalization of the pullback, we equip the family $\cX\to B$ with marked points $\sigma_\bullet$ such that the first $n-1$ of them are mapped to $p_1,\dotsc,p_{n-1}$, and the rest -- to the boundary divisor. In particular, we obtain a $1$-morphism $B\to \CM_{g,3d+n-1}$ that induces the family $(\cX\to B,\sigma_\bullet)$. Notice that the natural morphism $f\:\cX\to\PP^2$ satisfies Property (ii).

The compactification $\overline{\CM}_{g,3d+n-1}$ admits a finite surjective morphism from a projective scheme $\overline{M}$, cf. \cite[\S2.24]{dJ96}. Thus, after replacing $B$ with an irreducible component of $B\times_{\overline{\CM}_{g,3d+n-1}}\overline{M}$, we may extend the family $(\cX\to B,\sigma_\bullet)$ to a family of stable curves with marked points over a smooth projective base, i.e., we may assume that $B$ is projective. Furthermore, the family over $B$ is the pullback of the universal family over $\overline{M}$ along $B\to \overline{M}$. Plainly, the induced rational map $f\:\cX\dashrightarrow\PP^2$ still satisfies Property (ii).

Let $B_\st$ be the stable model of $B$, and $B^0$ the closure of the graph $B\to \overline{M}$ in $B_\st\times\overline{M}$. It is a projective integral model of $B$ over which the family $(\cX\to B,\sigma_\bullet)$ naturally extends. Furthermore, the extension is obtained by pulling back the universal family from $\overline{M}$. After replacing $B^0$ with a prestable model dominating $B^0$, whose reduction has smooth irreducible components, we obtain a model of the base over which the family extends to a family of stable marked curves satisfying Property (ii). It remains to achieve splitness.  To do so, we proceed as in \cite[\S~5.17]{dJ96}.

We begin by adding sections such that the nodes of the geometric fibers are contained in the sections. Such sections exist by \cite[Lemma~5.3]{dJ96}. We temporarily add these sections as marked points, replace $B$ by an appropriate finite covering, and construct a prestable model of the base over which the family of curves with the extended collection of marked points admits a stable model. By construction, the irreducible components of the geometric fibers of the new family are smooth. By applying \cite[Lemmata 5.2 and 5.3]{dJ96} once again, we may assume that the new family admits a tuple of sections such that for any geometric fiber, any component contains a section, and any node is contained in a section. This implies that the new family is split; cf. \cite[\S~5.17]{dJ96}. Finally, we remove the temporarily marked points $\sigma_l$'s and stabilize. The obtained family is still split.

To summarize, we constructed a projective curve with marked points $(B,\tau_\bullet)$, its integral model $B^0$, a family of marked curves $(\cX\to B,\sigma_\bullet)$, and a rational map $f\:\cX\dashrightarrow\PP^2$ that satisfy Properties (i) and (ii).

\subsubsection*{Step 3: The conclusion}
\begin{lem}\label{lem:keylemma}
Let $h\:\Gamma_\Lambda\to \RR^2$ be the tropicalization of $f\:\cX\dashrightarrow\PP^2$ with respect to $\cX^0\to B^0$. Then there exists a leg $l$ of $\Lambda$ such that the lengths of all but one edge of $\GG_l$ are constant in the family $\Gamma_\Lambda$, and the map $h$ is constant on all vertices of $\GG_l$, where $\GG_l$ is the underlying graph of the tropical curves parameterized by $l$.
\end{lem}

We postpone the proof of the lemma and first, deduce the theorem. Let $\tau\in B(K)$ be the marked point corresponding to the leg $l$ from Lemma~\ref{lem:keylemma}. By Corollary~\ref{cor:tropicalization_inf}, the geometric genus of $\cX_\tau$ is $g-1$. Furthermore,  the map $f$ is defined on $\cX_\tau$, maps its generic points to the dense orbit, and the pullback of the boundary divisor is a reduced divisor of degree $3d$. Thus, $f|_{\cX_\tau}$ is birational onto its image, and hence $[f(\cX_\tau)]\in Z$ is a reduced curve of genus $g-1$, which completes the proof of Theorem~\ref{thm:cont_lines}.\qed

\subsubsection*{Proof of Lemma~\ref{lem:keylemma}.}
To prove the lemma, we shall first analyze the image of the induced map $\alpha\:\Lambda\to M^\trop_{g,n-1,\nabla}$. The strata of the moduli space we will deal with admit no automorphisms. Thus, throughout the proof, one can think about $M^\trop_{g,n-1,\nabla}$ as a usual polyhedral complex rather than a generalized one. In particular, we will use the standard notion of {\em star} of a given stratum in a polyhedral complex.

Notice that $\alpha$ is not constant, since by the construction of $B$, for any $p\in \PP^2$, there exists $b\in B$ such that the curve $f(\cX_b)$ passes through $p$, and hence $h(\Gamma_{\trop(b)})$ contains $\trop(p)$, which can be any point in $\QQ^2\subset\RR^2$, cf. Remark~\ref{rem:pointwise tropicalization}. We will also need the following key properties of $\alpha$, which we prove next: (a) if $\alpha$ maps a vertex $v\in V(\Lambda)$ to a simple wall $M_\Theta$, see Definition~\ref{defn:simple wall}, then there exists a vertex $w\in V(\Lambda)$ such that $\alpha(w)=\alpha(v)$ and the map $\alpha$ is locally combinatorially surjective at $w$; and (b) if  $M_{\Theta'}$ is nice and $\alpha(\Lambda)\cap M_{\Theta'}\ne\emptyset$, then
$$\alpha(\Lambda)\cap M_{\Theta'}=M_{\Theta'}\cap\ev_{\Theta'}^{-1}(q_1,\dotsc,q_{n-1}),$$
where $\ev_{\Theta'}$ denotes the evaluation map defined in Section~\ref{sec:evaluation}. In particular, $\alpha(\Lambda)\cap M_{\Theta'}$ is an interval, whose boundary is disjoint from $M_{\Theta'}$, since so is $M_{\Theta'}\cap\ev_{\Theta'}^{-1}(q_1,\dotsc,q_{n-1})$ by Lemma~\ref{lem:dimbound} (3).

We start with property (a). Notice that $M_{\Theta}\cap\ev_\Theta^{-1}(q_1,\dotsc,q_{n-1})$ is a point by Lemma~\ref{lem:dimbound} (2). However, $\alpha$ is not constant, and hence there exists a vertex $w$ such that $\alpha(\Star(w))\nsubseteq M_\Theta$ and $\alpha(w)=\alpha(v)$. Thus, $\alpha$ is not harmonic at $w$, and hence it is locally combinatorially surjective at $w$ by Theorem~\ref{thm:paramtrop}.

Property (b) follows from harmonicity. Indeed, since $M_{\Theta'}\cap\ev_{\Theta'}^{-1}(q_1,\dotsc,q_{n-1})$ is an interval, $\emptyset\ne\alpha(\Lambda)\cap M_{\Theta'}\subseteq M_{\Theta'}\cap\ev_{\Theta'}^{-1}(q_1,\dotsc,q_{n-1})$, and $\alpha$ is affine on the edges and legs, it follows that $\alpha(\Lambda)\cap \overline{M}_{\Theta'}$ is a finite union of closed intervals. However, by Theorem~\ref{thm:paramtrop}, $\alpha$ is harmonic at the vertices $w\in V(\Lambda)$ that are mapped to $M_{\Theta'}$. Thus, $\alpha(\Lambda)\cap \overline{M}_{\Theta'}$ is a single closed interval, whose boundary is disjoint from $M_{\Theta'}$, which implies (b).

Pick a point $p_n\in\PP^2$ with tropicalization $q_n\in\RR^2$ such that the collection $\{p_i\}_{i=1}^{n}$ is in general position, and the configuration $\{q_i\}_{i=1}^{n}$ is vertically stretched, see Section~\ref{subsec:floordecomp}. Assume further that the points $\{q_i\}_{i=1}^{n}$ belong to a line defined by $y=-\mu x$ for some $\mu\gg 1$. Let $b\in B(K)$ be a point such that $f(\cX_b)$ passes through $p_n$. Then the tropicaliztion $(\Gamma,h)$ of $(\cX_b,f)$ is a floor decomposed curve by Proposition~\ref{prop:floor decomposed curves}. By Proposition~\ref{prop:genmikgen}, if we marked the point $p_n$ on $\cX_b$, the tropicalization $\trop(\cX_b; p_1,\dotsc,p_n)$ would be weightless and $3$-valent, and the map to $\RR^2$ would be an immersion away from the contracted legs. The tropicalization $\trop(\cX_b; p_1,\dotsc,p_{n-1})$ is obtained from $\trop(\cX_b; p_1,\dotsc,p_n)$ by removing the leg contracted to $q_n$ and stabilizing. Thus, $\alpha(\trop(b))$ belongs to a nice stratum $M_{\Theta'}$. Furthermore, the elevator $E$ adjacent to the top floor of $\Gamma$ has multiplicity one, cf. Remark~\ref{rem:multiplicity of elevator}.

Denote the floors from the bottom to the top by $F_1,\dotsc, F_d$, and let $F_k$, $k<d$, be the non-top floor adjacent to $E$. We may assume that $q_n$ belongs to the image of the elevator $E$, and $q_i$ to the image of the floor $F_i$ for all $1\le i\le d$. Indeed, pick a permutation $\sigma\:\{1,\dotsc, n\}\to\{1,\dotsc,n\}$ such that $q_{\sigma(n)}\in h(E)$ and $q_{\sigma(i)}\in h(F_i)$ for all $1\le i\le d$. Set $p'_i:=p_{\sigma(i)}$, and consider the curve $B'$ associated to $\{p'_i\}_{i=1}^{n-1}$ and the corresponding family of parameterized curves $f'\:\cX'\dashrightarrow\PP^2$. By construction, there exists $b'\in B'(K)$ such that $(\cX'_{b'},f'|_{\cX'_{b'}})=(\cX_b,f|_{\cX_b})$. It remains to replace $B$ with $B'$, $(\cX,f)$ with $(\cX',f')$, and $p_i$'s with $p'_i$'s.

Denote the $x$-coordinate of $E\cap F_k$ by $x$, and let $E'$ be the downward elevator adjacent to $F_k$ whose $x$-coordinate $x'$ is the closest to $x$. Notice that by Remark~\ref{rem:multiplicity of elevator}, the $x$-coordinate of the marked point $q_d$ does not belong to the interval joining $x$ and $x'$. Without loss of generality we may assume that $x'>x$. We will call a point of $F_k$ {\em special} if it is either a marked point or a vertex. Denote the $x$-coordinates of the special points of $F_k$ that belong to the interval $[x,x']$ by $x=x_0<x_1<\dots<x_r=x'$. If $x_i$ belongs to an elevator, then the elevator will be denoted by $E_i$. Plainly, $E=E_0$ and $E_r=E'$.

Set $q_n(t):=q_n+t(x_1-x_0,0)$, $0\le t\le 1$, and consider the continuous family of tropical curves $(\Gamma_t, h_t)\in \overline{M}_{\Theta'}\cap\ev_{\Theta'}^{-1}(q_1,\dotsc,q_{n-1})$, where $(\Gamma_0,h_0)=(\Gamma,h)$, and such that $q_n(t)\in h_t(E)$ for all $t$. Since $q_1,\dotsc, q_{n-1}, q_n(t)$ are vertically $\lambda$-stretched for a very large value of $\lambda$, the floors of the curve remain disjoint in the deformation $(\Gamma_t, h_t)$. Notice also that since the points $q_1,\dotsc, q_{n-1}$ are fixed, the $x$-coordinate of each elevator except $E$ remains the same in the deformation. Furthermore, for any $i\ne k,d$, the $x$-coordinates of all elevators adjacent to $F_i$ are fixed in the deformation, as well as the position of the marked point $q_i$ on $F_i$. Thus, the restriction of the parameterization $h$ to $F_i$ is also fixed. Finally, since the $x$-coordinate of $q_d$ does not belong to $[x,x']$, it follows that $(\Gamma_t, h_t)\in \overline{M}_{\Theta'}\cap\ev_{\Theta'}^{-1}(q_1,\dotsc,q_{n-1})$ for all $0\le t\le 1$. Furthermore, the curve $(\Gamma_1, h_1)$ belongs to a simple wall $M_\Theta$, in which the elevator $E=E_0$ gets adjacent to a $4$-valent vertex on the floor $F_k$ together with either the elevator $E_1$ or the leg contracted to the marked point $q_k$.
Let us describe $\Star(M_\Theta)$ explicitly. It consists of three nice strata $M_{\Theta'}, M_{\Theta''}, M_{\Theta'''}$, cf. Case~2 of Step~3 in the proof of Theorem~\ref{thm:paramtrop}. The stratum $M_{\Theta''}$ parameterizes curves in which the elevator $E=E_0$ has $x$-coordinate larger than that of $E_1$ (resp. $q_k$), and $M_{\Theta'''}$ parameterizes curves in which $E_0$ and $E_1$ (resp. $q_k$) get adjacent to a common vertex $u$ in the perturbation of the $4$-valent vertex of $(\Gamma_1,h_1)$; see Figure~\ref{fig:wall}. By Property (b), there exists a vertex $v$ of $\Lambda$, such that $\alpha(v)\in M^\trop_{g,n-1,\nabla}$ is the isomorphism class of $(\Gamma_1, h_1)$, and hence by Property (a), there exists a vertex $w$ such that $\alpha(w)=\alpha(v)$ and $\alpha$ is locally combinatorially surjective at $w$.

\tikzset{every picture/.style={line width=0.75pt}} 
\begin{figure}[ht]	
\begin{tikzpicture}[x=0.5pt,y=0.5pt,yscale=-.8,xscale=.8]
\import{./}{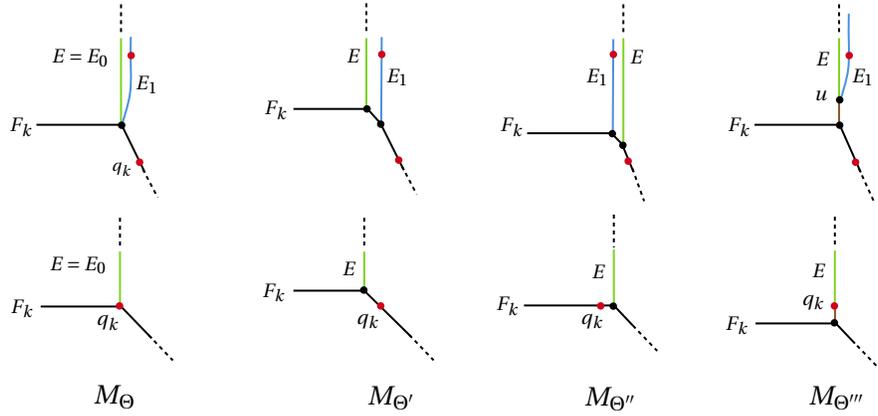}
\end{tikzpicture}		
\caption{On the left, the elevator $E$ reaches a special point. On the right -- the local pictures of typical elements in the three nice cones in $\Star(M_\Theta)$. The top row illustrates the case when the special point is a vertex, and the bottom -- when it is a marked point.}
\label{fig:wall}
\end{figure}

We proceed by induction on $(k,r)$ with the lexicographic order, and start with the extended base of induction: {\em $r=1$ and the multiplicity of $E'$ is one}. Since $\nabla$ is reduced, this is the case in the actual base of induction $(k,r)=(1,1)$. Let $e\in \Star(w)$ be an edge such that $\alpha(e)\subset M_{\Theta'''}$. Since the elevators $E$ and $E'$ both have multiplicity one, it follows from the balancing condition that the third edge adjacent to $u$ gets contracted by the parameterization map $h$; see Figure~\ref{fig:baseind}.

\tikzset{every picture/.style={line width=0.75pt}}
\begin{figure}[h]	
\begin{tikzpicture}[x=0.5pt,y=0.5pt,yscale=-.8,xscale=.8]
\import{./}{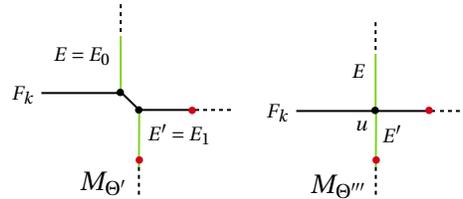}
\end{tikzpicture}		
\caption{The case in which the $4$-valent vertex is adjacent to two elevators of multiplicity one. The right picture illustrates the fact that in this case, the curves parameterized by $M_{\Theta'''}\cap\ev_{\Theta'''}^{-1}(q_1,\dotsc,q_{n-1})$ have constant image in $\RR^2$ and contain an edge of varying length contracted to $u$.}
\label{fig:baseind}
\end{figure}

Thus, the locus $M_{\Theta'''}\cap\ev_{\Theta'''}^{-1}(q_1,\dotsc,q_{n-1})$ is an {\em unbounded} interval. Furthermore, all curves in this locus factor through $(\Gamma_1,h_1)$. By Property (b), $M_{\Theta'''}\cap\ev_{\Theta'''}^{-1}(q_1,\dotsc,q_{n-1})=\alpha(\Lambda)\cap M_{\Theta'''}$. Hence there exists a leg $l\in L(\Lambda)$ such that $\alpha(l)\subset M_{\Theta'''}\cap\ev_{\Theta'''}^{-1}(q_1,\dotsc,q_{n-1})$ is not bounded. The leg $l$ satisfies the assertion of the lemma. Next, let us prove the induction step. We distinguish between two cases.

Case 1: {\em $r>1$.} Let $e\in \Star(w)$ be an edge such that $\alpha(e)\subset M_{\Theta''}$, and $(\Gamma_{1+\epsilon},h_{1+\epsilon})$ the tropical curve parameterized by any point in $e^\circ$. Then the corresponding invariant is $(k,r-1)$, and hence, by the induction assumption, there exists a leg $l\in L(\Lambda)$ satisfying the assertion of the lemma. See Figure~\ref{fig:wallcros1} for an illustration.

\tikzset{every picture/.style={line width=0.75pt}}
\begin{figure}[ht]	
\begin{tikzpicture}[x=0.5pt,y=0.5pt,yscale=-.7,xscale=.7]
\import{./}{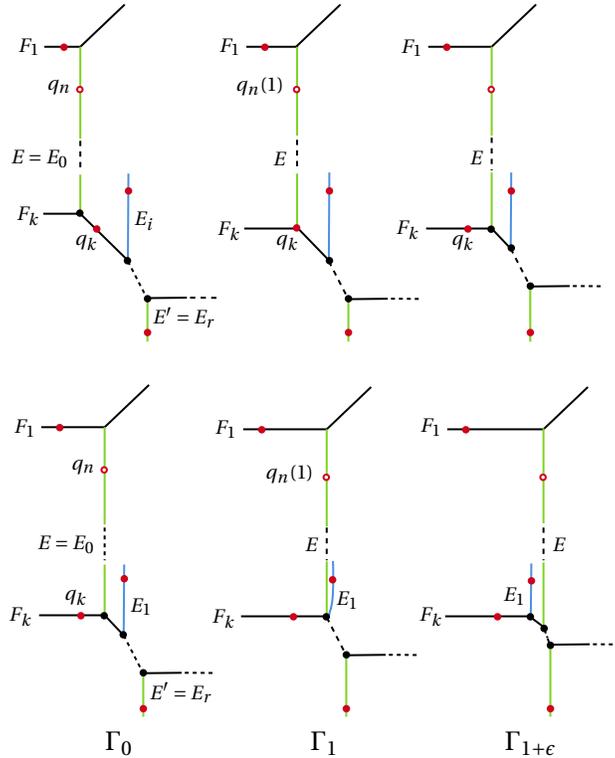}
\end{tikzpicture}		
\caption{Moving the elevator $E$ along the floor $F_k$. The top row illustrates the wall-crossing corresponding to a marked point, and the bottom -- to a downward elevator.}
\label{fig:wallcros1}
\end{figure}

Case 2: {\em $r=1$ and $k>1$.} We already checked the case when the multiplicity of $E'$ is one in the extended base of induction. Thus, assume that the multiplicity $\m(E')$ of $E'$ is greater than one. Let $q_j$ be the marked point contained in $h_1(E')$, and $E''$ the second elevator for which $q_j\in h_1(E'')$. Let $F_{k'}$ be the floor adjacent to $E''$. Then $k'<k$. Let $e\in \Star(w)$ be an edge such that $\alpha(e)\subset M_{\Theta'''}$, where $M_{\Theta'''}$ denotes the nice cone in the star of $M_\Theta$ in which the elevators $E$ and $E'$ get attached to a $3$-valent vertex $u$ obtained from the perturbation of the $4$-valent vertex of $(\Gamma_1,h_1)$. We denote the two elevators corresponding to $E'$ in this cone by $E'_1$ and $E'_2$, see Figure~\ref{fig:wallcros2} for an illustration.

\tikzset{every picture/.style={line width=0.75pt}}
\begin{figure}[ht]	
\begin{tikzpicture}[x=0.5pt,y=0.5pt,yscale=-0.7,xscale=0.7]
\import{./}{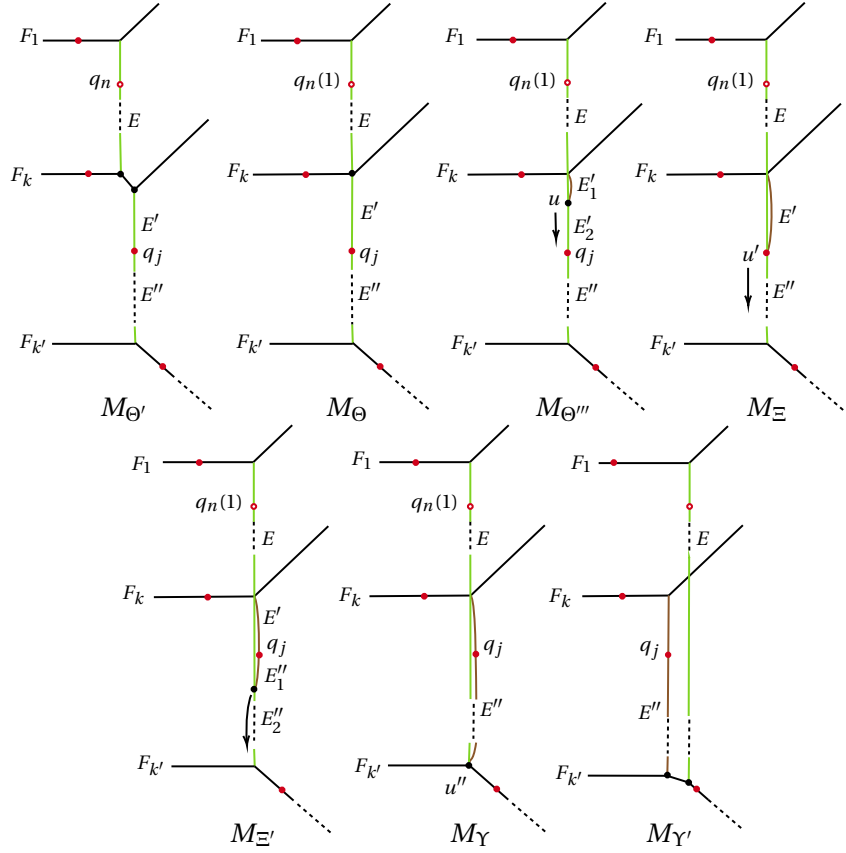}
\end{tikzpicture}		
\caption{Going down with an elevator $E'$ of multiplicity greater than one, and the corresponding wall-crossings.}
\label{fig:wallcros2}
\end{figure}

It follows from the balancing condition that the slope of $E'_1$ in $\Star(u)$ is $(0,\m(E')-1)$. Thus, the locus $M_{\Theta'''}\cap\ev_{\Theta'''}^{-1}(q_1,\dotsc,q_{n-1})$ is a bounded interval, whose second boundary point belongs to another simple wall $M_\Xi$, which parameterizes curves with a $4$-valent vertex $u'$ adjacent to the elevators $E,E',E''$ and to the leg $l_j$ contracted to $q_j$.

As before, there exists a vertex $w'$ of $\Lambda$ such that $\alpha(w')\in M_\Xi$ and $\alpha$ is locally combinatorially surjective at $w'$ by Property (b). Let $M_{\Xi'}\in\Star(M_\Xi)$ be the nice stratum in which the $4$-valent vertex $u'$ is perturbed to a $3$-valent vertex adjacent to $E$ and the downward elevator $E''_2$, and a $3$-valent vertex adjacent to $E'$ and $l_j$. By Property (a), the image of $\alpha$ intersects $M_{\Xi'}$ non-trivially. Furthermore, $M_{\Xi'}\cap \ev_{\Xi'}^{-1}(q_1,\dotsc,q_{n-1})$ is a bounded interval, whose second boundary point belongs to the simple wall $M_\Upsilon$ parameterizing curves with a $4$-valent vertex $u''$ adjacent to the floor $F_{k'}$ and to the elevators $E$ and $E''$. As usual, $M_{\Xi'}\cap \ev_{\Xi'}^{-1}(q_1,\dotsc,q_{n-1})$ belongs to the image of $\alpha$, and hence all three nice strata of $\Star(M_\Upsilon)$ intersect the image of $\alpha$ non-trivially. In particular, there exists an edge $e'$ of $\Lambda$ parameterizing curves with invariant $(k',*)<(k,1)$, and hence, by the induction assumption, there exists a leg $l\in L(\Lambda)$ satisfying the assertion of the lemma, which completes the proof.\qed

\section{The local geometry of Severi varieties}\label{sec:local geometry of vgd}
The goal of this section is to prove the following theorem describing the local geometry of $\overline V_{g,d}$ and $\overline V_{g,d}^{\irr}$ along $V_{1-d,d}$:

\begin{thm}\label{thm:branchstr}
The germ of $\overline V_{g,d}$ at a general $[C_0]\in V_{1-d,d}$ is a union of smooth branches indexed by subsets $\mu\subseteq C_0^\sing$ of cardinality $\delta:=\binom{d-1}{2}-g$. Moreover,
\begin{enumerate}
\item The scheme-theoretic intersection $\Br(\mu)\cap\Br(\mu')$ is smooth of codimension $|\mu'\setminus\mu|$ in $\Br(\mu)$;
\item The branch $\Br(\mu)$ belongs to $\overline V_{g,d}^{\irr}$ if and only if $C_0\setminus \mu$ is connected.
\end{enumerate}
\end{thm}

Throughout this section, we fix the degree $d\ge 1$, the genus ${1-d\le g\le \binom{d-1}{2}}$, and a general union of lines $[C_0]\in V_{1-d,d}$. Recall that for a projective algebraic curve $C$ and a point $p\in C$, the {\em $\delta$-invariant} of $C$ at $p$ is defined to be $\delta(C;p):=\dim\left((\CO_{C^\nu}/\CO_C)\otimes \CO_{C,p}\right)$, where $C^\nu$ denotes the normalization of $C$. Plainly, $\delta(C;p)\ne 0$ if and only if $p$ is singular. The total $\delta$-invariant of $C$ is then defined to be $\delta(C):=\dim\left(\CO_{C^\nu}/\CO_{C}\right)=\sum_{p\in C}\delta(C;p)$. Recall also, that $\delta(C)=\pa(C)-\pg(C)$ is the difference between the arithmetic and the geometric genera; see, e.g., \cite[Theorem~8]{Ros52}. We set $\delta:=\binom{d-1}{2}-g$, which is the $\delta$-invariant of an irreducible curve of degree $d$ and geometric genus $g$ in $\PP^2$. In order to analyze the local geometry of $\overline V_{g,d}$ along $V_{1-d,d}$, it is convenient to consider the decorated Severi varieties as introduced in \cite{Tyo07}.

\begin{defn}\label{defn:decorated severi variety}
The \emph{decorated Severi variety} is the incidence locus $U_{d,\delta}\subset |\mathcal O_{\mathbb P^2}(d)|\times (\mathbb P^2)^\delta$ consisting of the tuples $[C;p_1,\dotsc,p_\delta]$, in which $C$ is a reduced curve of degree $d$, and $p_1,\dotsc,p_\delta$ are distinct nodes of $C$. The union of the irreducible components $U\subseteq U_{d,\delta}$, for which $C$ is irreducible, is also called a decorated Severi variety, and is denoted by $U^\irr_{d,\delta}$.
\end{defn}

The following is a version of \cite[Proposition~2.11]{Tyo07}.

\begin{prop}\label{prop:decorated severi} Let $\mathbf{C}:=[C;p_1,\dotsc,p_\delta]\in U_{d,\delta}$ be a point, and $\phi\:|\mathcal O_{\mathbb P^2}(d)|\times (\mathbb P^2)^\delta\to |\mathcal O_{\mathbb P^2}(d)|$ the natural projection. Then,
\begin{enumerate}
\item The restriction of $d\phi$ to the tangent space $T_{\mathbf{C}}(U_{d,\delta})$ is injective;
\item The variety $U_{d,\delta}$ is smooth of pure dimension $\dim(U_{d,\delta})=\binom{d+2}{2}-1-\delta=3d+g-1$;
\item $\phi(\overline{U}_{d,\delta})=\overline{V}_{g,d}$, and $\left(\phi|_{\overline{U}_{d,\delta}}\right)^{-1}(\overline{V}_{g,d}^\irr)=\overline{U}_{d,\delta}^\irr$.
\end{enumerate}
\end{prop}

\begin{proof}
The proof is a rather straight-forward computation. After removing a general line from $\PP^2$, we get an open affine plane $\mathbb A^2=\Spec K[x,y]\subset \mathbb P^2$ containing all the $p_l$'s. We denote the coordinates $x,y$ on the $l$-th copy of $\PP^2$ in $(\PP^2)^\delta$ by $x_l,y_l$, and identify $H^0(\PP^2,\mathcal O_{\mathbb P^2}(d))$ with the space of polynomials of degree at most $d$ in the variables $x$ and $y$. Then the decorated Severi variety $U_{d,\delta}$ is given locally at $\mathbf{C}$ by the system of $3\delta$ (homogeneous in the coefficients of $F$) equations:
$$F(x_l,y_l)=F_x(x_l,y_l)=F_y(x_l,y_l)=0,$$
where $F=\sum_{i,j}a_{ij}x^iy^j\in H^0(\PP^2,\CO_{\PP^2}(d))$, and $F_x,F_y$ are its partial derivatives.

Let $G\in H^0(\PP^2,\CO_{\PP^2}(d))$ be a polynomial defining $C$, and for each $l$, let $(\lambda_l,\mu_l)$ be the coordinates of the point $p_l$. Then the tangent space $T_{\mathbf{C}}(U_{d,\delta})$ is given by the system of $3\delta$ equations
\begin{equation}\label{eq:syseqtan}
\forall l\;\;  \left\{
  \begin{array}{l}
    \sum_{i,j}\lambda_l^i\mu_l^jda_{ij}+G_x(p_l)dx_l+G_y(p_l)dy_l=0\\
    \sum_{i,j}i\lambda_l^{i-1}\mu_l^jda_{ij}+G_{xx}(p_l)dx_l+G_{xy}(p_l)dy_l=0 \\
    \sum_{i,j}j\lambda_l^i\mu_l^{j-1}da_{ij}+G_{xy}(p_l)dx_l+G_{yy}(p_l)dy_l=0
  \end{array}
\right.
\end{equation}
Notice that the kernel $\Ker(d\phi)$ is given by $da_{ij}=0$ for all $i$ and $j$. However, since $p_l\in C$ is a node for each $l$, the matrices
$$\left[\begin{matrix}
G_{xx}(p_l)&G_{xy}(p_l)\\G_{xy}(p_l)&G_{yy}(p_l)\end{matrix}\right]$$
are invertible, and $G_x(p_l)=G_y(p_l)=0$. Therefore, the intersection $T_{\mathbf{C}}(U_{d,\delta})\cap \Ker(d\phi)$ is zero, and assertion (1) follows. Furthermore, the dimension of the space of solutions of \eqref{eq:syseqtan} is equal to the dimension of the space of solutions of the following system of $\delta$ equations:
\begin{equation}\label{eq:syseqtan2}
\forall l\;\; \sum_{i,j}\lambda_l^i\mu_l^jda_{ij}=0,
\end{equation}
and if the latter system has full rank, then, by the Jacobian criterion, the variety  $U_{d,\delta}$ is smooth of dimension $\dim(U_{d,\delta})=\binom{d+2}{2}-1-\delta$.

Using the canonical identification $T_{[C]}(|\mathcal O_{\mathbb P^2}(d)|)\cong H^0(C,\CO_C(d))$, we conclude that the differential $d\phi$ induces an isomorphism
$T_{\mathbf{C}}(U_{d,\delta})\to H^0(C,\CI(d))\subset H^0(C,\CO_C(d))$, where $\CI$ is the ideal sheaf of the scheme of nodes $Z:=\bigcup_{l=1}^\delta p_l$. In order to proceed, we need the following Lemma, whose proof we postpone.

\begin{lem}\label{lem:h1=0}
$h^1(C,\CI(d))=0$.
\end{lem}

It follows from the lemma that the sequence
$$0\to H^0(C,\CI(d))\to H^0(C,\CO_C(d))\to H^0(Z,\CO_Z)\to 0$$
is exact. Thus,
$h^0(C,\CO_C(d))-h^0(C,\CI(d))=h^0(Z,\CO_Z)=\delta$, and hence the system \eqref{eq:syseqtan2} has full rank. Therefore, as explained above, $U_{d,\delta}$ is smooth of pure dimension $\binom{d+2}{2}-1-\delta$ as asserted in (2).

To prove (3), notice that the fibers of $\phi|_{U_{d,\delta}}$ are finite. Thus, $\phi(\overline{U}_{d,\delta})$ is pure dimensional of dimension $\binom{d+2}{2}-1-\delta=3d+g-1$. However, by the very definition, $\phi(U_{d,\delta})\subseteq \bigcup_{g'\le g}\overline{V}_{g',d}$, and $\dim(V_{g',d})=3d+g'-1<3d+g-1$ for any $g'<g$ by Proposition~\ref{prop:severi variety dimension} and Remark~\ref{rem:dimvgd}. Thus, $\phi(\overline{U}_{d,\delta})$ is a union of irreducible components of $\overline{V}_{g,d}$. On the other hand, by Corollary~\ref{cor:zariski}, any irreducible component $V\subseteq \overline{V}_{g,d}$ admits a dense open subset that parameterizes nodal curves, and hence belongs to $\phi(U_{d,\delta})$. Therefore, $\overline{V}_{g,d}\subseteq \phi(\overline{U}_{d,\delta})$, and hence $\phi(\overline{U}_{d,\delta})=\overline{V}_{g,d}$. The last assertion of (3) now follows from the definitions of $\overline{V}_{g,d}^\irr$ and $\overline{U}_{d,\delta}^\irr$.
\end{proof}

\begin{proof}[Proof of Lemma~\ref{lem:h1=0}]
The lemma is identical to \cite[Claim~2.12]{Tyo07}, and the proof given in {\em loc.cit.} works in arbitrary characteristic. For the completeness of presentation, we include a variation of this proof that proceeds by induction on the number of irreducible components of $C$.

 If $C$ is irreducible, consider the conductor ideal $\CI^{\rm cond}\subseteq \CO_{C}$, i.e., the annihilator of $\CO_{C^\nu}/\CO_C$, where $C^\nu\to C$ denotes the normalization of $C$. Notice that $\CI^{\rm cond}$ is an ideal also in $\CO_{C^\nu}$ under the natural embedding $\CO_C\subseteq \CO_{C^\nu}$. Thus, $H^i(C,\CI^{\rm cond}(d))=H^i(C^\nu,\CI^{\rm cond}(d))$ for all $i$, since the fibers of $C^\nu\to C$ are zero-dimensional. It follows from the definition that the vanishing locus of $\CI^{\rm cond}$ is $C^\sing$, and hence $\CI^{\rm cond}\subseteq\CI$. Consider the exact sequence of cohomology
$$H^1(C,\CI^{\rm cond}(d))\to H^1(C,\CI(d))\to H^1(C,\CI/\CI^{\rm cond}\otimes\CO_C(d)).$$
Since the quotient $\CI/\CI^{\rm cond}$ is a torsion sheaf, the group $H^1(C,\CI/\CI^{\rm cond}\otimes\CO_C(d))$ vanishes, and hence the map $H^1(C, \CI^{\rm cond}(d))\to H^1(C, \CI(d))$ is surjective. However, by \cite[Theorem~14]{Ros52}, the degree of the invertible sheaf $\CI^{\rm cond}(d)$ on $C^\nu$ is given by
$$c_1(\CI^{\rm cond}(d))=d^2-2\delta(C)=d^2+2p_g(C^\nu)-(d-2)(d-1)=3d+2g-2>2g-2,$$
and hence $h^1(C,\CI^{\rm cond}(d))=h^1(C^\nu,\CI^{\rm cond}(d))=0$, which implies $h^1(C, \CI(d))=0$.

To prove the induction step, let $C_1\subset C$ be an irreducible component of degree $d_1<d$, and $C_2$ the union of the other irreducible components. Set $d_2:=d-d_1$, and $Z_i:=(Z\cap C_i)\setminus C_{3-i}$ for $i=1,2$. Let $\CI_i\subset \CO_{C_i}$ be the ideal sheaf of $Z_i$ and $f_i\in H^0(C,\CO_C(d_i))$ a defining equation of $C_i$. Finally, denote by $\iota$ the closed immersion $C_1\hookrightarrow C$. Consider the exact sequence
\begin{equation}\label{eq:exactseq}
0\to \iota_*\CI_1(d_1)\to \CI(d)\to \CF(d)\to 0,
\end{equation}
where the first map is multiplication by $f_2$. Then $\CF$ is supported on $C_2$. Furthermore, it is the ideal sheaf of $Z_2$ union with the zero-dimensional scheme defined by the ideal $\langle f_1, f_2\rangle$. In particular, there is a natural embedding $\CI_2(d_2)\hookrightarrow \CF(d)$ given by multiplication by $f_1$, whose cokernel $\CG$ is a torsion sheaf. We conclude that there
is an exact sequence of cohomology $$H^1(C_2,\CI_2(d_2))\to H^1(C_2,\CF(d))\to H^1(C_2,\CG),$$ in which the first group vanishes by the induction assumption, and the last group vanishes since $\CG$ is a torsion sheaf. Thus, $H^1(C_2,\CF(d))=0$. Similarly, we get an exact sequence of cohomology $$H^1(C_1,\CI_1(d_1))=H^1\left(C,\iota_*\CI_1(d_1)\right)\to H^1(C,\CI(d))\to H^1(C,\CF(d))=0$$ associated to \eqref{eq:exactseq}. And again, the first group vanishes by the induction assumption. Therefore $h^1(C,\CI(d))=0$, which completes the proof.
\end{proof}

\begin{defn}
A {\em $\delta$-marking} on $C_0$ is an ordered subset $\vec\mu\subseteq C_0^\sing$ of cardinality $\delta$. A $\delta$-marking $\vec\mu$ is called {\em irreducible} if $C_0\setminus\vec\mu$ is connected.
\end{defn}

\begin{rem}
Since $\delta$ and $C_0$ are fixed, we will usually call $\delta$-markings on $C_0$ simply {\em markings}. For a marking $\vec\mu$ we denote the underlying set by $\mu$.
\end{rem}

\begin{proof}[Proof of Theorem~\ref{thm:branchstr}]
Let $\vec\mu=(p_1,\dotsc, p_\delta)$ be a marking. Consider the germ $U_{\vec\mu}$ of the decorated Severi variety $U_{d,\delta}$ at $[C_0;p_1,\dotsc, p_\delta]$. By Proposition~\ref{prop:decorated severi}, it is smooth of dimension $3d+g-1$, and $U_{\vec\mu}$ is isomorphic to its image under the natural projection $\phi$ to $|\mathcal O_{\mathbb P^2}(d)|$. Thus, $\phi(U_{\vec\mu})$ is a smooth branch of $\overline V_{g,d}$. Furthermore, the germ of $\overline V_{g,d}$ is covered by such branches. Plainly, $\phi(U_{\vec\mu})$ depends only on the underlying set $\mu$, and we set $\Br(\mu):=\phi(U_{\vec\mu})$. On the other hand, the set $\mu$ is determined by the branch $\Br(\mu)$. Indeed, a general curve $[C]\in \Br(\mu)$ corresponds to a nodal curve $C$ of geometric genus $g$ by Corollary~\ref{cor:zariski}. Thus, it has precisely $\delta$ nodes, and the set $\mu$ is the specialization of the set of nodes of $C$ to $C_0$. We conclude that the germ of $\overline V_{g,d}$ at $[C_0]$ consists of smooth branches indexed by subsets $\mu\subseteq C_0^\sing$ of cardinality $\delta$.

To prove (1), let $\mu=\{p_1,\dotsc,p_\delta\}$ and $\mu'=\{p'_1,\dotsc,p'_\delta\}$ be two subsets of $C_0^\sing$ of cardinality $\delta$. Set $r:=|\mu\cap\mu'|$. Without loss of generality we may assume that $p_i=p'_i$ for all $1\le i\le r$. Then the intersection $\Br(\mu)\cap\Br(\mu')$ contains $\Br(\mu\cup\mu')$, which is the image of the germ of $U_{d,2\delta-r}$ at $[C_0;p_1,\dotsc p_\delta,p'_{r+1},\dotsc,p'_\delta]$. Recall that $\Br(\mu\cup\mu')$ is smooth of dimension $\binom{d+2}{2}-1-(2\delta-r)$ by Proposition~\ref{prop:decorated severi}, and let us show that the scheme-theoretic intersection $\Br(\mu)\cap\Br(\mu')$ coincides with $\Br(\mu\cup\mu')$. To do so, we compare the tangent spaces.

Let $\CI\subseteq\CO_{C_0}$, $\CI'\subseteq\CO_{C_0}$, and $\CJ\subseteq\CO_{C_0}$ be the ideal sheaves of $\mu$, $\mu'$, and $\mu\cup\mu'$ respectively. Then the tangent space to $\Br(\mu)$ is given by
$T_{[C_0]}(\Br(\mu))=H^0(C_0,\CI(d))\subseteq H^0(C_0,\CO_{C_0}(d)),$
cf. the proof of Proposition~\ref{prop:decorated severi}. Similarly, the tangent space to $\Br(\mu')$ is given by
$T_{[C_0]}(\Br(\mu'))=H^0(C_0,\CI'(d))$, and to $\Br(\mu\cup\mu')$ by $T_{[C_0]}(\Br(\mu\cup\mu'))=H^0(C_0,\CJ(d))$. Thus,
$$T_{[C_0]}(\Br(\mu))\cap T_{[C_0]}(\Br(\mu'))=H^0(C_0,\CI'(d))\cap H^0(C_0,\CI(d))=H^0(C_0,\CJ(d))=T_{[C_0]}(\Br(\mu\cup\mu')),$$
which implies $\Br(\mu)\cap\Br(\mu')=\Br(\mu\cup\mu')$ scheme-theoretically. In particular, $\Br(\mu)\cap\Br(\mu')$ is smooth of dimension $\binom{d+2}{2}-1-(2\delta-r)$. However, by Proposition~\ref{prop:decorated severi}, $\Br(\mu)$ is smooth of dimension $\binom{d+2}{2}-1-\delta$, and hence the codimension of $\Br(\mu)\cap\Br(\mu')$ in $\Br(\mu)$ is $\delta-r=|\mu'\setminus\mu|$ as asserted.

To prove (2), notice that a general curve $C$ in a given branch $\Br(\mu)$ is irreducible, if and only if its normalization $C^\nu$ is so. But $C^\nu$ specializes to the partial normalization of $C_0$ at $\mu$, and hence $C^\nu$ is irreducible if and only if the partial normalization of $C_0$ at $\mu$ is connected. The latter is clearly equivalent to $C^0\setminus \mu$ being connected.
\end{proof}

\section{The proof of the Main Theorem}\label{sec:proofMT}

The Main Theorem follows from the following stronger assertion about the irreducibility of decorated Severi varieties:

\begin{thm}\label{thm:irrdecsev}
The decorated Severi variety $U_{d,\delta}^\irr$ is either empty or irreducible. Furthermore, it is empty if and only if $\delta>\binom{d-1}{2}$.
\end{thm}

Indeed, set $\delta=\binom{d-1}{2}-g$. Then $\delta\le\binom{d-1}{2}$ since $g\ge 0$. By Proposition~\ref{prop:decorated severi} (3), $\overline V_{g,d}^\irr$ is the image of $\overline U_{d,\delta}^\irr$ under the natural projection to $|\mathcal O_{\mathbb P^2}(d)|$. By Theorem~\ref{thm:irrdecsev}, $U_{d,\delta}^\irr$ is non-empty and irreducible, and hence so is $V_{g,d}^\irr$. Finally, by Proposition~\ref{prop:severi variety dimension}, the Severi variety $V_{g,d}$ is either empty or equidimensional of dimension $3d+g-1$. Thus, $\dim(V_{g,d}^\irr)=3d+g-1$, which completes the proof of the Main Theorem. \qed

\begin{proof}[The proof of Theorem~\ref{thm:irrdecsev}]

If $\delta>\binom{d-1}{2}$, then $U_{d,\delta}^\irr=\emptyset$ since the geometric genus of any plane curve of degree $d$ with $\delta$ nodes is at most $\frac{d-1}{2}-\delta<0$, and hence none of such curves is irreducible. Vice versa, assume that $\delta\le\binom{d-1}{2}$. Then $\delta\le\binom{d-1}{2}=\binom{d}{2}-(d-1)=|C_0^\sing|-(d-1)$. Pick an irreducible component $L$ of $C_0$. Then the number of nodes of $C_0$ that belong to $L$ is $d-1$, and hence there exists a marking $\vec\mu$ disjoint from $L$. Since $L$ intersects any other component of $C_0$, it follows that $\vec\mu$ is an irreducible marking. Thus, by Theorem~\ref{thm:branchstr} (2), $V_{g,d}^\irr$ is not empty, and hence so is $U_{d,\delta}^\irr$ by Proposition~\ref{prop:decorated severi} (3). We conclude that $U_{d,\delta}^\irr=\emptyset$ if and only if $\delta>\binom{d-1}{2}$.

Assume now that $U_{d,\delta}^\irr$ is not empty, and let us show that it is irreducible. We start by defining an equivalence relation on the set of markings. Let $\fG$ be the full group of symmetries of the set $C_0^\sing$ of $\binom{d}{2}$ nodes of $C_0$. The group $\fG$ acts naturally on the set of markings, and we use the notation $\vec\mu^\sigma$ to denote the image of a marking $\vec\mu$ under the action of $\sigma\in\fG$. Let $\vec\mu$ be a marking, and $L,L',L''$ three different components of $C_0$. Set $p:=L\cap L'$, $q:=L'\cap L''$, and $r:=L\cap L''$, and let $\tau\in \fG$ be the transposition switching $q$ and $r$. If $p\notin\mu$, then we say that $\vec\mu^\tau$ is {\em similar} to $\vec\mu$. Plainly, $\vec\mu$ is irreducible if and only if so is $\vec\mu^\tau$. We define the equivalence relation $\sim$ on the set of markings to be the minimal equivalence relation for which any pair of similar markings is equivalent. Since similarity is a symmetric relation, the equivalence relation $\sim$ is nothing but the transitive closure of the similarity relation. Theorem~\ref{thm:irrdecsev} now follows from the following two lemmata -- the first is purely combinatorial, and the second is algebraic.

\begin{lem}\label{lem:eqirmar}
Any two irreducible markings are equivalent.
\end{lem}

\begin{lem}\label{lem:eqmarkcorto1comp}
If the markings $(p_1,\dotsc,p_\delta)$ and $(p'_1,\dotsc,p'_\delta)$ are equivalent, then $[C_0;p_1,\dotsc,p_\delta]$ and $[C_0;p'_1,\dotsc,p'_\delta]$ belong to the same irreducible component of $U_{d,\delta}$.
\end{lem}

Indeed, by the lemmata, there exists a unique irreducible component $U\subseteq U_{d,\delta}$ that contains $[C_0;p_1,\dotsc,p_\delta]$ for any irreducible marking $\vec\mu=(p_1,\dotsc,p_\delta)$. If $U'\subseteq U_{d,\delta}^\irr$ is any irreducible component, then the projection $\phi(U')$ of $U'$ to $|\mathcal O_{\mathbb P^2}(d)|$ is dense in an irreducible component of $\overline V^\irr_{g,d}$ by Proposition~\ref{prop:decorated severi}, and $C_0\in\phi(\overline{U'})$ by Theorem~\ref{thm:cont_lines}. Thus, there exists a marking $\vec\mu=(p_1,\dotsc, p_\delta)$ such that $[C_0;p_1,\dotsc,p_\delta]\in U'$. Since $U'\subseteq U_{d,\delta}^\irr$, the marking $\vec\mu$ is irreducible by Theorem~\ref{thm:branchstr} (2) and Proposition~\ref{prop:decorated severi} (3). Hence $U'=U$, which completes the proof.
\end{proof}

\begin{proof}[Proof of Lemma~\ref{lem:eqirmar}]
Although the lemma is a particular case of \cite[Lemma~3.10]{Tyo07}, we include its proof for the convenience of the reader. Denote the irreducible components of $C_0$ by $L_1,\dotsc, L_d$, and set $q_{i,j}:=L_i\cap L_j$. Let $\vec\mu$ be an irreducible marking.

First, let us show that $\vec\mu$ is equivalent to a marking disjoint from $L_1$. If $\mu\cap L_1\ne\emptyset$, then the union $C$ of the components $L$ of $C_0$ for which $L\cap L_1\in\mu$ is not empty. Let $C'$ be the union of the remaining components excluding $L_1$. Since $C_0\setminus\mu$ is connected, $C'\ne\emptyset$, and there exists $r\in C\cap C'$ that does not belong to $\mu$. Let $L_i\subseteq C$ and $L_j\subseteq C'$ be the components containing $r$. Then $q_{1,i}\in\mu, q_{1,j}\notin\mu$, and hence $\vec\mu\sim\vec\mu^\tau$, where $\tau$ is the transposition switching $q_{1,i}$ and $r$. Since $r\notin\mu$ and $q_{1,i}\in\mu$, it follows that $|\mu^\tau\cap L_1|<|\mu\cap L_1|$. Thus, repeating this process finitely many times, we can find a marking disjoint from $L_1$ and equivalent to $\vec\mu$. It remains to show that if $\vec\mu$ and $\vec\mu'$ are disjoint from $L_1$, then $\vec\mu\sim\vec\mu'$.

Assume that $\vec\mu$ is disjoint from $L_1$. Let $\tau$ be a transposition switching a pair of distinct nodes in $C_0^\sing\setminus L_1$. It is sufficient to prove that $\vec\mu\sim\vec\mu^\tau$. Indeed, such transpositions generate the full symmetric group of nodes $C_0^\sing\setminus L_1$. Therefore, if $\vec\mu'$ is disjoint from $L_1$, then $\vec\mu'=\vec\mu^\sigma$ for some permutation $\sigma$ of $C_0^\sing\setminus L_1$, and hence $\vec\mu\sim\vec\mu'$. Let us prove that $\vec\mu\sim\vec\mu^\tau$.

Assume first, that there exist indices $i,j,k>1$ such that $\tau$ switches the nodes $q_{i,j}$ and $q_{i,k}$. Let $\tau_j$ be the transposition switching $q_{i,j}$ and $q_{i,1}$, and $\tau_k$ the transposition switching $q_{i,k}$ and $q_{i,1}$. Since $q_{1,j},q_{1,k}\notin\mu$, it follows that $\vec\mu$ is equivalent to $\vec\mu^{\tau_j\tau_k\tau_j}$. However, $\tau_j\tau_k\tau_j=\tau$, which implies the claim. Assume now, that there are no $i,j,k$ as above. Then there exist four distinct indices $i,j,k,l>1$ such that $\tau$ switches $q_{i,j}$ and $q_{k,l}$. Let $\tau_i$ be the transposition switching $q_{i,j}$ and $q_{i,l}$, and $\tau_l$ be the transposition switching $q_{i,l}$ and $q_{k,l}$. Then $\vec\mu\sim\vec\mu^{\tau_i\tau_l\tau_i}$ by the first case considered above. However, $\tau=\tau_i\tau_l\tau_i$, and we are done.
\end{proof}

\begin{proof}[Proof of Lemma~\ref{lem:eqmarkcorto1comp}]
As before, we denote the irreducible components of $C_0$ by $L_1,\dotsc, L_d$, and set $q_{i,j}:=L_i\cap L_j$. It is sufficient to prove the lemma for similar markings. Set $\mu:=(p_1,\dotsc,p_\delta)$, and let $\tau$ be the transposition for which $\mu^\tau=(p'_1,\dotsc,p'_\delta)$. Without loss of generality we may assume that $q_{1,2}\notin\mu$, and $\tau$ is the transposition switching $q_{1,3}$ and $q_{2,3}$. Set $\delta':=\binom{d}{2}-1$, and pick a $\delta'$-marking $\vec\mu'=(p_1,\dotsc,p_{\delta'})$ extending $\mu$ such that $\mu'=C_0^\sing\setminus\{q_{1,2}\}$. It is enough to show that there exists an irreducible component $U\subset U_{d,\delta'}$ containing both $[C_0;p_1,\dotsc,p_{\delta'}]$ and $[C_0;\tau(p_1),\dotsc,\tau(p_{\delta'})]$. Indeed, consider the natural forgetful map $\psi\:U_{d,\delta'}\to U_{d,\delta}$. Then $\psi(U)\subseteq U_{d,\delta}$ is an irreducible subvariety that contains both $\mathbf{C}_0:=[C_0;p_1,\dotsc,p_{\delta}]$ and $\mathbf{C}_0^\tau:=[C_0;\tau(p_1),\dotsc,\tau(p_{\delta})]$, and hence $\mathbf{C}_0$ and $\mathbf{C}_0^\tau$ belong to the same irreducible component of $U_{d,\delta}$, as asserted.

Let us show that $U$ as above exists. To simplify the notation, we assume that $\delta=\delta'$, and reorder the points so that
$(p_1,\dotsc,p_\delta)=\left(q_{1,3},\dotsc,q_{1,d},q_{2,3},\dotsc,q_{2,d},q_{3,4},\dotsc,q_{3,d},\dotsc,q_{d-1,d}\right).$
In particular, $q_{1,3}=p_1$ and $q_{2,3}=p_{d-1}$. Then $\tau$ is the transposition on $C_0^\sing$ switching $p_1$ and $p_{d-1}$. Notice that the action of $\tau$ on $\{\mathbf{C}_0, \mathbf{C}_0^\tau\}$ extends to an action on $U_{d,\delta}$ that switches the first marked point with the $(d-1)$-st, and hence induces a permutation on the set of the irreducible components of $U_{d,\delta}$. Let $U$ be the irreducible component containing  $\mathbf{C}_0$, and let us show that $\mathbf{C}_0^\tau\in U$.

By Proposition~\ref{prop:decorated severi}, the variety $U_{d,\delta}$ has pure dimension $2d+1$. In particular, $\dim(U)=2d+1$. Thus, by dimension count, a general $\mathbf{C}:=[C;r_1,\dotsc,r_{\delta}]\in U$ is a union of a general conic $Q$ and $d-2$ lines $L'_3,\dotsc, L'_d$ in general position. Furthermore, $Q\cap L'_i=\{r_{i-2},r_{d+i-4}\}$ for all $3\le i\le d$, and $L'_3\cap L'_i=\{r_{2d+i-7}\}$ for all $4\le i\le d$. We fix general $Q,L'_4,\dotsc,L'_d$ as above, and set $D:=(\cup_{i\ge 4}L'_i)\cup Q$.

Consider the dense open subset $W\subset (\PP^2)^*$ parameterizing the lines $L'\subset\PP^2$ that intersect $D$ transversally and such that $L'\cap D^\sing=\emptyset$. Let $\widetilde{W}$ be the incidence variety parameterizing tuples $[L'; r'_1,r'_{d-1},r'_{2d-3},\dotsc,r'_{3d-7}]$, where $r'_1,r'_{d-1}$ are the points of intersection of $L'$ with $Q$, and $r'_{2d+i-7}$ is the point of intersection of $L'$ with $L'_i$ for all $4\le i\le d$. Then $\widetilde{W}\to W$ is an \'etale covering of degree two. Furthermore, there is a natural morphism $\iota\:\widetilde{W}\to U_{d,\delta}$, that maps $[L'; r'_1,r'_{d-1},r'_{2d-3},\dotsc,r'_{3d-7}]$ to $$[L'\cup D; r'_1,r_2,\dotsc, r_{d-2},r'_{d-1},r_d,\dotsc r_{2d-4},r'_{2d-3},\dotsc,r'_{3d-7},r_{3d-6},\dotsc r_\delta].$$

It is easy to see that $\widetilde{W}$ is irreducible. Indeed, consider the natural projection $\widetilde{W}\to Q^2$ mapping $[L'; r'_1,r'_{d-1},r'_{2d-3},\dotsc,r'_{3d-7}]$ to $(r'_1,r'_{d-1})$. Since the line $L'$ is uniquely determined by the pair of points $(r'_1,r'_{d-1})$, the variety $\widetilde{W}$ is mapped bijectively onto its image in $Q^2$, and hence the image has dimension $\dim(\widetilde{W})=\dim(W)=\dim\left((\PP^2)^*\right)=2=\dim(Q^2)$. Thus, $\widetilde{W}\to Q^2$ is dominant, and bijective onto its image. Hence $\widetilde{W}$ is irreducible since so is $Q^2$. To finish the proof, it remains to notice that both $\mathbf{C}$ and $\mathbf{C}^\tau$ belong to $\iota(\widetilde{W})$, and hence belong to the same irreducible component of $U_{d,\delta}$. But $\mathbf{C}\in U$ and $\mathbf{C}^\tau\in U^\tau$, which implies $U=U^\tau$. Thus, $\mathbf{C}_0^\tau\in U^\tau=U$, and we are done.
\end{proof}

\providecommand{\bysame}{\leavevmode\hbox to3em{\hrulefill}\thinspace}
\providecommand{\MR}{\relax\ifhmode\unskip\space\fi MR }
\providecommand{\MRhref}[2]{%
  \href{http://www.ams.org/mathscinet-getitem?mr=#1}{#2}
}
\providecommand{\href}[2]{#2}


\begin{thebibliography}{CCUW20}

\bibitem[AC81a]{AC81}
Enrico Arbarello and Maurizio Cornalba, \emph{Footnotes to a paper of
  {B}eniamino {S}egre}, Math. Ann. \textbf{256} (1981), no.~3, 341--362.

\bibitem[AC81b]{AC81it}
\bysame, \emph{Su una congettura di {P}etri}, Comment. Math. Helv. \textbf{56}
  (1981), 1--37.

\bibitem[ACGS20]{ACGS}
Dan Abramovich, Qile Chen, Mark Gross, and Bernd Siebert, \emph{Decomposition
  of degenerate {G}romov-{W}itten invariants}, Compos. Math. \textbf{156}
  (2020), no.~10, 2020--2075.

\bibitem[ACP15]{ACP}
Dan Abramovich, Lucia Caporaso, and Sam Payne, \emph{The tropicalization of the
  moduli space of curves}, Ann. Sci. \'{E}c. Norm. Sup\'{e}r. (4) \textbf{48}
  (2015), no.~4, 765--809.

\bibitem[BIMS15]{BIMS}
Erwan Brugall\'{e}, Ilia Itenberg, Grigory Mikhalkin, and Kristin Shaw,
  \emph{Brief introduction to tropical geometry}, Proceedings of the
  {G}\"{o}kova {G}eometry-{T}opology {C}onference 2014, G\"{o}kova
  Geometry/Topology Conference (GGT), G\"{o}kova, 2015, pp.~1--75.

\bibitem[BM09]{BM08}
Erwan Brugall\'{e} and Grigory Mikhalkin, \emph{Floor decompositions of
  tropical curves: the planar case}, Proceedings of {G}\"{o}kova
  {G}eometry-{T}opology {C}onference 2008, G\"{o}kova Geometry/Topology
  Conference (GGT), G\"{o}kova, 2009, pp.~64--90.

\bibitem[BPR13]{BPR}
Matthew Baker, Sam Payne, and Joseph Rabinoff, \emph{On the structure of
  non-{A}rchimedean analytic curves}, Tropical and non-{A}rchimedean geometry,
  Contemp. Math., vol. 605, Amer. Math. Soc., Providence, RI, 2013,
  pp.~93--121.

\bibitem[CC99]{CC99}
L.~Chiantini and C.~Ciliberto, \emph{On the {S}everi varieties on surfaces in
  {${\mathbb P}^3$}}, J. Algebraic Geom. \textbf{8} (1999), no.~1, 67--83.

\bibitem[CCUW20]{CCUW}
Renzo Cavalieri, Melody Chan, Martin Ulirsch, and Jonathan Wise, \emph{A moduli
  stack of tropical curves}, Forum Math. Sigma \textbf{8} (2020), Paper No.
  e23.

\bibitem[CHL06]{CL06}
Hung-Jen Chiang-Hsieh and Joseph Lipman, \emph{A numerical criterion for
  simultaneous normalization}, Duke Math. J. \textbf{133} (2006), no.~2,
  347--390.

\bibitem[CHT22a]{CHT20b}
Karl Christ, Xiang He, and Ilya Tyomkin, \emph{Degeneration of curves on some
  polarized toric surfaces}, J. Reine Angew. Math. \textbf{787} (2022),
  197--240.

\bibitem[CHT22b]{CHT21b}
\bysame, \emph{Irreducibility of {S}everi varieties and {H}urwitz schemes in
  small characteristic}, In preparation, 2022.

\bibitem[Del85]{D85}
Pierre Deligne, \emph{Le lemme de {G}abber}, Ast\'{e}risque (1985), no.~127,
  131--150, Seminar on arithmetic bundles: the Mordell conjecture (Paris,
  1983/84).

\bibitem[dJ96]{dJ96}
A.~J. de~Jong, \emph{Smoothness, semi-stability and alterations}, Inst. Hautes
  \'{E}tudes Sci. Publ. Math. (1996), no.~83, 51--93.

\bibitem[dJHS11]{JHS}
A.~J. de~Jong, Xuhua He, and Jason~Michael Starr, \emph{Families of rationally
  simply connected varieties over surfaces and torsors for semisimple groups},
  Publ. Math. Inst. Hautes \'{E}tudes Sci. (2011), no.~114, 1--85.

\bibitem[DM69]{DM69}
P.~Deligne and D.~Mumford, \emph{The irreducibility of the space of curves of
  given genus}, Inst. Hautes \'{E}tudes Sci. Publ. Math. (1969), no.~36,
  75--109.

\bibitem[DPT80]{Tei80}
Michel Demazure, Henry~Charles Pinkham, and Bernard Teissier (eds.),
  \emph{S\'{e}minaire sur les {S}ingularit\'{e}s des {S}urfaces}, Lecture Notes
  in Mathematics, vol. 777, Springer, Berlin, 1980, Held at the Centre de
  Math\'{e}matiques de l'\'{E}cole Polytechnique, Palaiseau, 1976--1977.

\bibitem[EC85]{EC18}
Federigo Enriques and Oscar Chisini, \emph{Lezioni sulla teoria geometrica
  delle equazioni e delle funzioni algebriche. 1. {V}ol. {I}, {II}}, Collana di
  Matematica [Mathematics Collection], vol.~5, Zanichelli Editore S.p.A.,
  Bologna, 1985, Reprint of the 1915 and 1918 editions.

\bibitem[Ful69]{Ful69}
William Fulton, \emph{Hurwitz schemes and irreducibility of moduli of algebraic
  curves}, Ann. of Math. (2) \textbf{90} (1969), 542--575.

\bibitem[GM07]{GM07}
Andreas Gathmann and Hannah Markwig, \emph{The numbers of tropical plane curves
  through points in general position}, J. Reine Angew. Math. \textbf{602}
  (2007), 155--177.

\bibitem[Har86]{Har86}
Joe Harris, \emph{On the {S}everi problem}, Invent. Math. \textbf{84} (1986),
  no.~3, 445--461.

\bibitem[HM82]{HM82}
Joe Harris and David Mumford, \emph{On the {K}odaira dimension of the moduli
  space of curves}, Invent. Math. \textbf{67} (1982), no.~1, 23--88, With an
  appendix by William Fulton.

\bibitem[Hur91]{Hur91}
A.~Hurwitz, \emph{\"{U}ber {R}iemann'sche {F}l\"{a}chen mit gegebenen
  {V}erzweigungspunkten}, Math. Ann. \textbf{39} (1891), no.~1, 1--60.

\bibitem[Ill71]{Ill71}
Luc Illusie, \emph{Complexe cotangent et d\'{e}formations. {I}}, Lecture Notes
  in Mathematics, Vol. 239, Springer-Verlag, Berlin-New York, 1971.

\bibitem[Kle82]{Klein}
Felix Klein, \emph{\"{U}ber {R}iemann's {T}heorie der algebraischen
  {F}unctionen und ihrer {I}ntegrale.}, Teubner, Leipzig, 1882.

\bibitem[Knu83]{Knu83}
Finn~F. Knudsen, \emph{The projectivity of the moduli space of stable curves.
  {III}. {T}he line bundles on {$M_{g,n}$}, and a proof of the projectivity of
  {$\overline M_{g,n}$} in characteristic {$0$}}, Math. Scand. \textbf{52}
  (1983), no.~2, 200--212.

\bibitem[Kon95]{Kon95}
Maxim Kontsevich, \emph{Enumeration of rational curves via torus actions}, The
  moduli space of curves ({T}exel {I}sland, 1994), Progr. Math., vol. 129,
  Birkh\"{a}user Boston, Boston, MA, 1995, pp.~335--368.

\bibitem[KS13]{KS13}
Steven~L. Kleiman and Vivek~V. Shende, \emph{On the {G}\"{o}ttsche threshold},
  A celebration of algebraic geometry, Clay Math. Proc., vol.~18, Amer. Math.
  Soc., Providence, RI, 2013, With an appendix by Ilya Tyomkin, pp.~429--449.

\bibitem[Lan20]{lang20}
Lionel Lang, \emph{Monodromy of rational curves on toric surfaces}, J. Topol.
  \textbf{13} (2020), no.~4, 1658--1681.

\bibitem[Mik05]{Mik05}
Grigory Mikhalkin, \emph{Enumerative tropical algebraic geometry in
  {$\mathbb{R}^2$}}, J. Amer. Math. Soc. \textbf{18} (2005), no.~2, 313--377.

\bibitem[Mor79]{Mori79}
Shigefumi Mori, \emph{Projective manifolds with ample tangent bundles}, Ann. of
  Math. (2) \textbf{110} (1979), no.~3, 593--606.

\bibitem[Ran22]{Ran19}
Dhruv Ranganathan, \emph{Logarithmic {G}romov-{W}itten theory with expansions},
  Algebr. Geom. \textbf{9} (2022), no.~6, 714--761.

\bibitem[Rau17]{rau17}
Johannes Rau, \emph{A first expedition to tropical geometry},
  \url{https://www.math.uni-tuebingen.de/user/jora/downloads/FirstExpedition.pdf},
  2017.

\bibitem[Ros52]{Ros52}
Maxwell Rosenlicht, \emph{Equivalence relations on algebraic curves}, Ann. of
  Math. (2) \textbf{56} (1952), 169--191.

\bibitem[Sev21]{Severi}
Francesco Severi, \emph{Vorlesungen \"{u}ber algebraische {G}eometrie, {A}nhang
  {F}}, Teubner, Leipzig, 1921.

\bibitem[{Sta}20]{stacks-project}
The {Stacks project authors}, \emph{The stacks project},
  \url{https://stacks.math.columbia.edu}, 2020.

\bibitem[Tyo07]{Tyo07}
Ilya Tyomkin, \emph{On {S}everi varieties on {H}irzebruch surfaces}, Int. Math.
  Res. Not. IMRN (2007), no.~23, Art. ID rnm109, 31.

\bibitem[Tyo12]{Tyo12}
\bysame, \emph{Tropical geometry and correspondence theorems via toric stacks},
  Math. Ann. \textbf{353} (2012), no.~3, 945--995.

\bibitem[Tyo13]{Tyo13}
\bysame, \emph{On {Z}ariski's theorem in positive characteristic}, J. Eur.
  Math. Soc. (JEMS) \textbf{15} (2013), no.~5, 1783--1803.

\bibitem[Tyo14]{Tyo14}
\bysame, \emph{An example of a reducible {S}everi variety}, Proceedings of the
  {G}\"{o}kova {G}eometry-{T}opology {C}onference 2013, G\"{o}kova
  Geometry/Topology Conference (GGT), G\"{o}kova, 2014, pp.~33--40.

\bibitem[Zar82]{Zar82}
Oscar Zariski, \emph{Dimension-theoretic characterization of maximal
  irreducible algebraic systems of plane nodal curves of a given order {$n$}
  and with a given number {$d$} of nodes}, Amer. J. Math. \textbf{104} (1982),
  no.~1, 209--226.

\end{thebibliography}
\end{document}